\documentclass{amsart} \usepackage{amssymb,amsfonts,amscd}

\newtheorem{theorem}{Theorem}[section]
\newtheorem{lemma}[theorem]{Lemma}
\newtheorem{corollary}[theorem]{Corollary}
\newtheorem{proposition}[theorem]{Proposition}
\theoremstyle{remark}
\newtheorem{remark}[theorem]{Remark}
\theoremstyle{definition}
\newtheorem{definition}[theorem]{Definition}

\numberwithin{equation}{section}

\def\s01{{\{0,1\}}}
\def\cosec #1 {{\textrm {cosec}\ #1}}

\def\seqq#1#2#3#4{{(#4)_{#1=#2}^{#3}}}
\def\summ#1#2#3{{\sum_{#1=#2}^{#3}}}

\def\prodd#1#2#3{{\prod_{#1=#2}^{#3}}}

\def\range #1#2#3{{{#1}={#2},\ldots,{#3}}}

\def\b1{{\bf 1} }

\def\del{{\delta}}
\def\Del{{\Delta}}

\def\alp{{\alpha}}

\def\bet{{\beta}}

\def\gam{{\gamma}}
\def\lam{{\lambda}}

\def\bN{{\ensuremath{\mathbb N} }}
\def\bC{{\ensuremath{\mathbb C} }}
\def\bZ{{\ensuremath{\mathbb Z} }}

\def\bs{{\bf s}}

\def\bi{{\bf i}}

\def\b #1{{\bf #1}}

\def\proof{{\beginpf}}
\def\endproof{{\eopf}}
\def\Lam{{\Lambda}}
\def\eps{{\varepsilon}}

\def\lin{{\textrm{lin}}}
\def\linbar{{\overline{\textrm{lin}}}}

\newcommand\beq{\begin{equation}}
\newcommand\eeq{\end{equation}}
\newcommand\bdfn{\begin{defn}}
\newcommand\blem{\begin{lemma}}
\newcommand\elem{\end{lemma}}
\newcommand\bcor{\begin{cor}}
\newcommand\ecor{\end{cor}}
\newcommand\bthm{\begin{thm}}
\newcommand\ethm{\end{thm}}
\newcommand\edfn{\end{defn}}
\newcommand\bcas{\begin{cases}}
\newcommand\ecas{\end{cases}}

\def\nm#1{{\left\Vert #1 \right\Vert}}
\def\Nm{\Vert \cdot \Vert}

\def\rad{\textrm {rad}}

\newcommand\sprod[2]{\langle{#1},{#2}\rangle}

\def\casif#1!{& \textrm{#1}\\}
\newif\ifrough
\roughfalse
\def\reff#1{\ifrough{\ref {#1}=#1}\else{\ref {#1}}\fi}
\def\l#1!{{\ifrough{\textrm{=#1\ }\label{#1}}\else{\label{#1}}\fi}}

\def\refeq#1{{(\reff{#1})}}
\def\up#1{^{(#1)}}

\def\reflem#1{{Lemma \reff{#1}}}
\def\refdfn#1{{Definition \reff{#1}}}
\def\refthm#1{{Theorem \reff{#1}}}
\def\refcor#1{{Corollary \reff{#1}}}
\def\tnm#1{{\left\|\hskip -1pt\left|#1\right|\hskip -1pt\right\|}}
\def\Tnm{{\tnm\cdot}}

\def\benum{\begin{enumerate}}
\def\eenum{\end{enumerate}}

\def\ap1{{\square}}

\newcounter{smallarabics}

\newcounter{smallalphs}

\newcommand\bproof{\begin{proof}}
\newcommand\eproof{\end{proof}}
\newcommand{\beginpf}{\smallskip\textbf{Proof. }}
\newcommand{\eopf}{\hfill $\Box$}

\def\cI{{\mathcal I}}

\def\opnmn#1{{{\left\Vert #1 \right\Vert}_{\textrm{op}}\up n}}
\def\nmtn#1{{{\left\Vert #1 \right\Vert}_{2}\up n}}
\def\Opnmn{{\opnmn\cdot}}

  \DeclareMathOperator{\spec}{Sp}

\DeclareMathOperator{\Cdb}{{\mathbb C}}

\DeclareMathOperator{\Ndb}{{\mathbb N}}

\begin{document}

\title[Operator algebras with cai in $c_0$]{Operator algebras with contractive approximate identities: A
large operator algebra in $c_0$} 
\author{David P. Blecher}
\author{Charles John Read}

\address{Department of Mathematics, University of Houston, Houston, TX
77204-3008}
 \email{dblecher@math.uh.edu}
\address{Department of Pure Mathematics,
University of Leeds,
Leeds LS2 9JT,
England}
 \email{read@maths.leeds.ac.uk}
\thanks{*Blecher was partially supported by a grant  from, 
the National Science Foundation.  Read is grateful for support from UK research council
grant  EP/K019546/1}

\begin{abstract}   
We exhibit a singly generated, semisimple commutative
operator algebra with a contractive approximate identity, 
such that the spectrum of the generator is a null sequence and zero, but the 
algebra is not the closed linear span of the idempotents associated with the null sequence
and obtained from the analytic functional calculus.  Moreover the multiplication 
on the algebra is neither compact nor weakly compact.  Thus we 
construct a `large' operator algebra of orthogonal idempotents, which may 
be viewed as a dense subalgebra of $c_0$. 
 \end{abstract}

\maketitle

\section{Introduction}    

There is an extensive history and theory of operator algebras
on a Hilbert (or Banach) space that are generated by a family of idempotent operators which are
orthogonal (that is, the product of any
two of which is zero); and using such families in `spectral resolutions' of operators. 
Related to this, it is well known that there exist exotic Banach algebras whose 
elements are sequences of scalars, with the 
multiplication being the obvious pointwise one.   That is, there can be quite complicated
Banach algebra norms on subalgebras of the $C^*$-algebra $c_0$  of null sequences. 
 However examples of both of these kinds of algebras, of a certain interesting type
 described below, and which have an approximate identity, 
seem to be missing from the Banach algebra and operator theory literature.   Our main goal here is to provide an explicit, yet 
in some sense universal, example of this kind.  
We first discuss our goal from the operator-theoretic angle, and later in the introduction we 
will mention the Banach algebra viewpoint. 

   Henceforth, by an {\em operator algebra}, we will
mean a norm closed subalgebra of $B(H)$, where the latter 
denotes the bounded linear operators on a Hilbert space
$H$.    There is a large literature 
on  operator algebras generated by a family of mutually orthogonal 
idempotents (see e.g.\ 
\cite{Az,DS,Hus, LN,Ricker} and references therein).  
Such a family arises naturally
when one considers for a bounded operator 
$T$ on a Hilbert space $H$ with Sp$(T)$ 
(countable and) having no nonzero limit  points, the spectral idempotents obtained by the analytic 
functional calculus
from the nonzero isolated points.
These idempotents will be called {\em minimal spectral idempotents}, and they are nonzero by a 
basic property of the functional calculus
(and even have the uniqueness property
in the Shilov idempotent theorem, see e.g. 2.4.33 in
\cite{Dales}).  In this case it is standard  to try to use this family 
to analyze the structure of $T$ (often with an eye to decomposing $T$ in terms of these 
idempotents).   We will henceforth assume that the norm closed algebra $B_T$ generated by 
$T$ in $B(H)$ is semisimple, which implies
that these minimal spectral idempotents $e$ are minimal in the sense that
$e B_T = \Cdb e$.   
It seems that certain specific questions about the algebra $B_T$ that arise in this setting, 
 are essentially unaddressed in the 
literature, to the best of our knowledge.   It is a simple exercise in matrix theory (and using the blanket assumption 
that $B_T$ is semisimple)
 that, if $H$ above were finite dimensional, then $T$ is the sum of the minimal spectral idempotents, each multiplied
by the corresponding eigenvalue.  So it is natural to ask if, for  $T$  as above, 
 $B_T$ is generated  by the minimal spectral idempotents of $T$?    (Saying that 
$B$ is `generated' by a subset, will for us always mean, unless 
stated to the contrary, that $B$ 
 is the smallest norm closed subalgebra of $B$ containing the subset.)  
Such questions  become quite difficult if one adds the assumption that the operator algebras involved have  
 approximate identities.
We will give a counterexample to the question a few lines above, with algebras 
possessing
contractive approximate identities (or {\em cai}'s).   In this example, 
$B_T$ is `large', and in particular is not {\em weakly compact}. 
We will define, for the purposes of this paper,
a commutative Banach algebra $A$ to be weakly compact (resp.\ compact)
if  multiplication on $A$ by $a$ is 
weakly compact (resp.\ compact), for every $a \in A$ (this is not the usual definition, but it is
equivalent to it for algebras with a cai).
 Indeed it is the case that for an operator $T$ with Sp$(T)$ 
 having no nonzero limit  points,
$B_T$ {\em is} generated (in the norm topology) by the minimal spectral idempotents 
of $T$ if $B_T$ is  semisimple,
weakly compact, and 
its socle has an approximate identity (see the end of 
this Introduction  for this and some related  
results, which we show there to be `best possible' in some sense).  

 In algebraic language,  $B_T$  being 
generated by the minimal spectral idempotents, is equivalent to  $B_T$ having 
{\em dense socle} (or being {\em  Tauberian} \cite{LN}); such algebras can be considered to 
be not `big'.  
However the point is that 
the examples in the literature  of operator 
algebras generated by an operator 
and an associated sequence of mutually orthogonal minimal idempotents,
 tend to either be `small' in this sense, or to fall within 
the case where this sequence is uniformly
bounded, or to not have approximate identities,
and do not speak to the question we have mentioned.  We remark that in joint work with Le Merdy,
the first author studied some natural Banach sequence algebras which 
were shown to be operator algebras (see \cite[Chapter 5]{BLM}), but again these 
were not `big' in the sense above, and had no approximate identity.  

We now discuss our goal from
the Banach algebraic perspective, which goes back to Kaplansky  (e.g.\ \cite{Kap}).
We recall that a  {\em natural Banach sequence algebra} on $\Ndb$ is a Banach algebra $A$ of scalar sequences,
which contains the space $c_{00}$ of
finitely supported sequences, and whose characters (i.e.\ nontrivial multiplicative linear functionals) are
precisely the obvious ones: $\chi_n(\vec a) = a_n$ for $\vec a \in A$ (see 
\cite[Section 4.1]{Dales} and \cite{DU}).    In our case 
the sequences in $A$ will converge to $0$, so that $c_{00} \subset A \subset c_0$ (we are not assuming of course that
the norm on $A$ is the $c_0$ norm).   Natural Banach  sequence algebras
have been studied by many researchers (see e.g.\ \cite{Dales,DU,LN} and references therein),
however we are not aware of any such algebras
in the literature which are `big' in the previous sense, namely that the  socle of $A$,
which in this case is $c_{00}$, 
 is not dense in $A$ (that is, $A$ is not Tauberian), or, more generally, that $A$ 
is not  weakly compact,
and which also have a bounded approximate identity (or {\em bai}).
An example due to Joel Feinstein  
of a natural Banach sequence algebra without a dense socle is given in 
\cite[Section 4.1]{Dales}, however it has no bai.  
 The existence of an example of this kind with a bai was asked of us by Dales.  We will construct here a `big' example, which is a singly generated  operator algebra on a Hilbert space, and which has a cai.   
 To relate the Banach  sequence algebra setting to that 
of the previous paragraphs, 
note that if $T$ is an operator on a Hilbert space $H$ with Sp$(T)$
 having no nonzero limit  points, and if the closed algebra  $B_T$ generated 
by $T$ is semisimple, then the Gelfand transform makes $B_T$ into a 
(semisimple) natural Banach sequence algebra  in $c_0$
(by basic Gelfand theory, e.g.\ the standard ideas in the proof of Corollary  \ref{finok}  below).

In this paper we exhibit an  operator algebra example with the desired features discussed above. 
 In hopes of obtaining a tool useful for solving other questions in this
area in the literature, we have deliberately built our example to be as `large as possible',
and this has probably added to the difficulty and complexity of our proofs.
 In particular we show:

\begin{theorem} \label{ch} There exists a  semisimple operator algebra $A$  
  which has a cai  and  a single generator $g$ (and hence $A$  is separable), with the following properties.  The spectrum of the generator of $A$ is
 a null sequence and zero, but $A$ is strictly larger than $\overline{A_{00}}$, the norm closed 
linear span of the minimal spectral idempotents
 associated with  this null sequence and obtained from the analytic functional calculus.  
Also, multiplication by the generator $g$ 
on $A$ is not a weakly compact operator (equivalently,
$A$ is not an ideal in $A^{**}$ with the Arens product, see e.g.\ \cite[1.4.13]{Pal}
or \cite[Lemma 5.1]{ABR}).  
The algebra can be chosen further with $\overline{A_{00}}$ having a cai too; and with
either $A$ contained in the strong operator closure of 
$\overline{A_{00}}$, or not, as desired.
\end{theorem}

Our example (in particular our algebras $A, \overline{A_{00}}$, and their unitizations), will hopefully be useful in settling other open questions in the subject.  As an illustration of how it can be used in that capacity, 
we mention that Joel Feinstein has pointed out to us that the unitization of our example 
 solves an old question of his (see \cite{JF1,JF2}, although he has informed us that the question goes back at least as far  his thesis), and another similar question of Dales, namely 
 whether a certain variant of the notion of peak sets for regular Banach function algebras are `sets of
 synthesis'.  We shall explain this application in more detail at the end of 
Section \ref{Addpro}.
 
Concerning the layout of our paper, in the next section we turn to the construction of our algebras $A$ and $\overline{A_{00}}$
described above.  The development will become increasingly technical
as the paper proceeds.  However in Section \ref{hilr} we will prove the 
key part of our main theorem
with one lemma taken on faith, and in Section \ref{Addpro} we will pause
and describe many properties that our algebras $A$ and $\overline{A_{00}}$
 possess.   
  The material following  Section \ref{Addpro}
consists of the
lengthy proof of the lemma  just referred to.

We end this introduction with some general positive results on the topics above;
namely  sufficient conditions
for when $B_T$ is generated by the minimal spectral idempotents of $T$.
We remark that in \cite[Proposition 1.1]{Az} 
 it is shown that any closed algebra on a separable Hibert space
generated by a family of `mutually orthogonal' idempotents, is 
topologically singly generated. 
  
\begin{proposition} \label{isalg}
   Suppose that $D$ is a Banach  algebra, 
which  is an essential  ideal in a commutative Arens regular Banach algebra $A$ (`essential' means that
the canonical representation of $A$ on $D$ is one-to-one).  Assume that  $A$ is
 weakly compact and $D$ has a bai.  Then $A = D$.  
\end{proposition}

\proof  
By e.g.\ \cite[1.4.13]{Pal} or \cite[Lemma 5.1]{ABR} we have $A^{**} A \subset A$.
Let $e \in D^{\perp \perp}$ be the `support projection' in $A^{**}$ of $D$, an identity for
$D^{\perp \perp}$, and write $1$ for the identity
of the unitization of $A$.
Then $(1-e) D = 0$, and  for any $a \in A$ we have $a(1-e) D = 0$.  Since
$e A \subset A$, and $D$ is an essential ideal in $A$, we see that $a(1-e) = 0$.
Thus $e$ acts as an identity on $A$ and therefore also on $A^{**}$, so that
$A^{**} = e A^{**} \subset
D^{\perp \perp}.$
Hence $A = D$.
\endproof

\begin{corollary} \label{finok}  Let $T$ be an operator on a Hilbert space whose spectrum
 has no nonzero limit points, and let $B_T$ (resp.\ $B$) 
be the closed algebra  generated by $T$
 (resp.\  by the minimal spectral idempotents of $T$).  If $B_T$ is semisimple
and weakly compact, and if 
$B$ has a bai,
then $B = B_T$.
\end{corollary}
\proof  Suppose that ${\rm Sp}(T) \setminus \{0 \} = \{ \lambda_n \}$.
By basic Gelfand theory, the set of characters of the unitization $B_T^1$ is $\{ \chi_n : n \in \Ndb_0 \}$,
where $\chi_0$ annihilates $B_T$, and $\chi_n(T) = \lambda_n$ for $n \in \Ndb$.  By the functional
calculus $\chi_m(e_n) = \delta_{nm}$ for $n \in \Ndb, m \in \Ndb_0$, where $e_n$ is the spectral idempotent 
in $B_T^1$ corresponding to $\lambda_n$.    Hence $e_n \in B_T$, and $e_n T - \lambda_n e_n \in 
\cap_{m} \, {\rm Ker}(\chi_m)= (0)$.  Thus $e_n B_T = \Cdb e_n$ for all $n$, 
and $\chi_n(a) e_n = a e_n$ for all $a \in B_T$.  
Hence $B B_T \subset B$,  that is $B$ is an ideal in $B_T$.  
 Indeed $B$ 
is an essential  ideal in $B_T$, since the latter is semisimple (if $a e_n = 0$ for all
$n$ then $\chi(a) = 0$ for all characters $\chi$).   Thus
$B = B_T$ by  
Proposition \ref{isalg}.   \endproof

\medskip

For operators on a Hilbert space whose spectrum
 has no nonzero limit points, the last result is sharp
 in the following sense:

\begin{theorem} \label{finok2}
  In the last result no one of the following three
 hypotheses 
can simply be
removed, in general: $B_T$ is semisimple; or  $B_T$ is  weakly compact; or 
$B$ has a bai.
\end{theorem}

\proof  To see that the semisimplicity condition cannot  be removed, consider the long example in \cite[Section 5]{BRI}
(or one could consider the
Volterra operator, or the direct sum
of the Volterra operator and a generator for $c_0$).
Theorem \ref{ch} in the present
paper shows that  the weak compactness condition cannot  be removed, even if in addition $B$ and
$B_T$ have cai.

Finally, we will show that
the approximate identity condition cannot be removed, even if the algebra is `compact'.
Our  algebra $A$ will be the space $c$ of convergent sequences with product 
$\vec x \cdot \vec y = (\frac{1}{2^n} \, x_n y_n)$ (an example also
mentioned briefly by Mirkil in a Banach algebra context).  
Clearly $A$
 is  generated by 
$T = (1,1, 1 \cdots)$ and $c_{00}$.  Since $\vec x \cdot \vec y$ is the usual
product of $\vec x, \vec y,$ and $(\frac{1}{2^n} )$,  $A$ is a commutative  operator algebra by \cite[Remark 2 on p. 194]{BRS}.
The vectors $2^n \vec e_n$ are minimal idempotents in $A$, inducing characters $\chi_n$ on $A$, and it is clear now
that $A$ is semisimple.   Conversely, since
 any character on $c_{0}$ must be induced by a sequence in $\ell^1$, it is easy to see that
such a character is the restriction of one of the $\chi_n$.  It is also then easy that any character on $A$ is
one of the $\chi_n$.   
We leave it as an exercise that
 $T$ generates $c_{00}$.  Hence  the spectrum of $T$ is $\{ \frac{1}{2^n} \} \cup \{ 0 \}$
and the spectrum of $A$ is homeomorphic to $\Ndb$.   If $E_n$ is the minimal spectral idempotent of $T$ corresponding 
to  $\frac{1}{2^n}$ in the spectrum, then $E_n \in A$ by an argument in Corollary \ref{finok}.  
  Thus
these minimal spectral idempotents
are exactly the $2^n \vec e_n$ above, by e.g.\ \cite[Theorem 1.2]{Kol}.  So $A = B_T$ has discrete spectrum, but it
is clearly not Tauberian.  To see that $A$ is compact suppose that $(\vec x(n))_n$ is a bounded sequence in $c$.
Then $T \cdot \vec x(n) = (x(n)_m/2^m)$, which  is the product of a fixed sequence in $c_0$ with
 a bounded sequence in $c_0$.
Since $c_0$ is compact, there is a convergent subsequence as desired.
 \endproof

\begin{remark}
Theorem \ref{finok2} suggests the
question if in Corollary \ref{finok} the condition that $B$ has a bai may be replaced by $B_T$ having a bai (or cai).
\end{remark}

\begin{corollary} \label{Palm}  A semisimple, topologically 
singly generated operator algebra,
whose socle has a bai,
has a  dense socle if and only if it is compact.
  \end{corollary} 

\proof  We first prove that, more generally, a semisimple, weakly compact,  operator algebra with a
topological single generator, 
whose socle has a bai,
and which has discrete spectrum, has a dense socle (equivalently, is Tauberian). This  follows
from Corollary \ref{finok}.  Indeed if $T$ is a topological single generator for 
such an algebra $A$, then 
the spectrum of $g$ (minus $0$) is homeomorphic to the discrete spectrum
of $A$, so has no nonzero limit points.   The minimal idempotents in $A$, whose span 
defines the socle, 
define characters of $A$, so, as in the proof of
Corollary \ref{finok}, they must be the minimal spectral idempotents 
of $T$.   

That a Banach algebra with dense socle is `compact' is obvious
(or see e.g.\ \cite[Proposition 8.7.7]{Pal}).  Conversely,  semisimple 
compact Banach algebras
 have discrete spectrum (see e.g.\ \cite[Chapter 8]{Pal}).   Thus the 
first result follows from the last paragraph. 
 \endproof

\begin{remark}
We point out an error in \cite{ABR} that momentarily led us astray early in this work.
Namely, in the last assertion of \cite[Theorem 5.10 (4)]{ABR}, to get a correct statement
the characters there are not allowed to vanish on
$A$.  This  led to a mistaken comment at the end of
the first paragraph of the Remark after Proposition 5.6 there, concerning the spectrum of $A^{**}$.
Fortunately these results have not been used elsewhere.
\end{remark}

\section{The general construction} \label{gco}

Let $A_0$ be the dense subalgebra of $c_0$ generated by the unit vectors $e_i$ and the vector
$g=\summ i1\infty \, 2^{-i} \, e_i  = (\frac 12,\frac 14,\frac 18,\ldots)$. 
(In passing we remark that this notation differs from 
the meaning of $A_0$ in \cite{Dales,DU}, and we apologize for 
any confusion to those familiar with that literature.) 
We seek to renorm $A_0$ so that
 its completion $A$ is  an operator algebra such that  
\begin{enumerate}
\item $A$ has a cai, and is topologically generated by $g$.
\item the spectrum of $g$ is $\{2^{-n}:n\in\bN\}$,
\item $g\notin{\overline\lin\{e_i:i\in\bN\}}$, and 
\item $A$ is semisimple.
\end{enumerate}
Note that what forces us to change (increase!) the usual norm on $c_0$ is condition (3).
In fact the norm will be increased in such a way that the unit vectors $e_i$, which will be the spectral idempotents for 
the generator $g$, are unbounded.
 For $n\in\bN_0$, write  $P_n=\summ i1ne_i$, and 
let $$g_n=2^ng\wedge 1=(1,1,1,\ldots,1,\frac 12,\frac 14,\frac 18,\ldots)
=\summ i1ne_i+\summ i1\infty2^{-i}e_{i+n} .$$ 
Note that $g_n \in 2^n g + \lin \{e_1, \cdots , e_n \} \in A_0.$

\blem\l 1!  Let $\Nm$ be any algebra norm on the algebra $A_0$.
Suppose that for some strictly increasing sequence  
  $\seqq i1\infty{a_i}\subset\bN$, 
we have 
$$\nm{g_{{a}_n}^{a_n}}\le 1+\frac 1n \; , \; \; \text{and} \; \nm{g_{{a}_n}^{a_n}\cdot g - g}\le \frac 1n.$$ 
Then 
the vectors $x_n=\frac n{n+1}g_{{a}_n}^{a_n}$ are a cai for $(A_0,\Nm)$.
\elem
\begin{proof}
The conditions we are given ensure that $\nm{x_n}\le 1$ and $x_n g\to g$. But $e_m=2^mge_m$, and so  
$x_n e_m\to e_m$ also. The vectors $e_m$ and $g$ generate $A_0$, and so $\seqq n1\infty{x_n}$ is a cai.
\end{proof}

\medskip

If $A$ is the completion of $A_0$ in some algebra norm, we write $\overline{A_{00}}$ for the closed ideal ${\overline \lin\{e_i:i\in\bN\}}$ in $A$.

\blem\l 2! Once again, let $\Nm$ be any algebra norm on the algebra $A_0$, and let $A$ denote the completion of $(A_0,\Nm)$.
Suppose that for some strictly increasing  sequence   $\seqq i1\infty{a_i}\subset\bN$, 
 we have 
$$\nm{g^{a_n}(I-P_{{a}_n})}\le n^{-a_n}. $$
Then 
the spectrum of $g\in A$ is precisely $\{2^{-n}:n\in\bN\}\cup\{0\}$,
and the spectral idempotents for the eigenvalues $2^{-n}$, obtained from the analytic functional calculus
for $g$, are precisely the unit vectors $e_n$.
\elem
\begin{proof}
For every $x\in A_0$ and $n\in\bN$ we have $xe_n=\chi_n(x) \, e_n$ for a unique 
complex number $\chi_n(x)$. Even in the completion $A$ this will be true, because 
for $y\in A$ the product $ye_n$ is a limit of scalar multiples of $e_n$, and so is a multiple of $e_n$.
We now see that $\chi_n$  is a character on $A$ with  $\chi_n(g)=2^{-n}$.  Thus  $2^{-n}\in\spec(g)$.

Conversely, we claim that $\spec(g)$ does not contain any $\lam\ne 0$ that is not a negative integer power of 2.
If it does, it contains such a $\lam\in\partial\spec(g)$, so $\lam$ is in the approximate point spectrum of $g$.
 Pick $n$ so large that $1/n<|\lam|$. For $y\in (I-P_{{a}_n})A$ with $\nm y=1$ we have  by hypothesis that 
$\nm{g^{a_n}y}\le n^{-a_n}<|\lam|^{a_n}$.  Thus $\lambda$ is not in the spectrum of the operator of multiplication
by $g$ on $(I-P_{{a}_n})A$.    There is therefore an $\eta>0$ such that $\nm{gy-\lam y}\ge\eta\nm y$
for all $y\in(I-P_{{a}_n})A$.  In the subalgebra $$P_mA=(e_1+e_2+\ldots e_m)A=\lin\{e_j:j\le m\}$$ the spectrum of $g$ is 
$\{2^{-j}:j\le m\}$, so there is an $\eta'$ such that    $\nm{gy'-\lam y'}\ge\eta'\nm {y'}$ for $y'\in P_mA$.
Write $m={a}_n$ and let $z\in A$.
Then $$\nm{gz-\lam z}\ge\frac 1{\nm{P_{m}}}\cdot\nm{gP_mz-\lam P_mz}\ge \eta'\nm{P_mz}/{\nm{P_{m}}},$$ and also
$$\nm{gz-\lam z}\ge\frac 1{1+\nm{P_{m}}}\cdot\nm{g(1-P_m)z-\lam(1-P_m)z}\ge \eta\nm{(1-P_m)z}/(1+\nm{P_{m}}).$$
Also $\nm{gz-\lam z}\ge C \nm{P_m z + (z - P_m z) } = C \nm{z}$, for some positive constant $C$.
Therefore $\lam$ is not in the approximate point spectrum of $g$. This contradiction proves the claim;
  that is, 
 $\spec(g) = \{ 0 \} \cup \{ 2^{-n} : n \in \bN \}$.

To identify the spectral idempotent for the point $2^{-n}$,
we decompose the unitization $A^1$ as a direct sum of ideals
$P_{a_m} A^1 \, \oplus \, (I-P_{a_m}) A^1$, where $m$ is much larger than $2^n$.
There is a corresponding decomposition of the spectral idempotent as a sum of the 
spectral idempotent in $P_{a_m} A^1$, which is easy to see is $e_n$, and 
 the spectral idempotent in $(I-P_{a_m}) A^1$.  The latter
 is zero  since by hypothesis, $\nm{g^{a_m}(I-P_{{a}_m})}\le m^{-a_m}$, so that
the spectral radius of $g(I-P_{{a}_m})$ is much smaller than $2^{-n}$.  (One may also use this
in conjunction with the criteria in  \cite[Theorem 1.2]{Kol}).
\end{proof}

\begin{corollary} \label{3}
  If the conditions of  {\rm  \reflem{2}} hold, then $A$ is singly generated by $g$, and 
the space of characters of $A$ is $\{\chi_n:n\in\bN\}$, with 
$\chi_n$ as defined above. 
\end{corollary} 
\begin{proof}  
 The spectral idempotents  $e_n$ 
are, by the functional calculus, 
in the closed algebra ${\rm oa}(g)$ generated by 
$g$ (note that $e_n \in {\rm oa}(1,g)$ by e.g.\ \cite[Theorem 2.4.4 (ii)]{Dales},
but $e_n = 2^n g e_n$, so that $e_n$ is in ${\rm oa}(g)$).   
Together with $g$ itself, these idempotents generate $A_0$ algebraically; so $A$ is singly generated by $g$.
As in the proof of Corollary \ref{finok}  
the characters  $\chi_n$ constitute the character space of $A$. 
\end{proof}

\begin{remark}
It is easy to find operator algebra norms so that (1)--(4) 
at the start of Section \ref{gco} hold 
with the exception of $A$ having a cai.
The last example in  the proof of Theorem \ref{finok2} is as such.
Also the operator algebra norm in \cite[Example 4.30]{ABS} can easily be seen to work (with the help of the last two results).
\end{remark}

\section{Maximal norms}
\l e!  

\begin{definition} \label{s0} 
 Let a strictly increasing sequence $({a}_n)_{n=1}^\infty\subset\bN$ be given.  Set $a_0 = 1$, and define a subset $S_0\subset c_0$ as follows:
\end{definition} 
\begin{equation} \label{s0e} S_0=\{g, a_n^{-1}e_n, \frac n{n+1}g_{{a}_n}^{a_n},  n(g_{{a}_n}^{a_n}\cdot g-g), n^{a_n}g^{a_n}(1-P_{{a}_n}):n\in\bN\}.
  \end{equation} 
 Also, let $S$ be the collection of all finite products of elements of $S_0$.

Note: later on, we are going to impose growth conditions on this underlying sequence.
Our main results will happen provided that the underlying sequence $(a_n)$ increases sufficiently rapidly.
But for now, we note that $S_0$ includes both $g$ and nonzero multiples of the unit vectors $e_n$, 
so $S_0$ generates $A_0$ algebraically; indeed the  linear span of $S$ is $A_0$.
Thus, we get a finite seminorm on $A_0$ if we define 
\begin{equation} \label{4-1}
  {\nm x}_{\max} =\inf\{\summ i1n|\lam_i|: x=\summ i1n\lam_is_i: n\in\bN, \lam_i\in\bC, s_i\in S\}.  
\end{equation}
It will eventually transpire  that, given growth conditions on the $a_n$,
the completion of $A_0$ in the norm $\Nm_{\max}$ gives a Banach algebra satisfying 
(1)--(3) at the start of Section \ref{gco}, and that the quotient of this by the radical 
satisfies (1)--(4) there.  

\blem\l 4! Let $S$, $S_0$ be defined as in \refdfn{s0}. Then the seminorm  $\Nm_{\max}$ in {\rm \refeq{4-1}} is a norm greater than or equal to the $c_0$-norm $\Nm_0$.   Indeed 
$\Nm_{\max}$ is the largest seminorm on $A_0$ such that $\nm s_{\max}\le 1$ for all $s\in S$; 
and $\Nm_{\max}$ is equal to the largest algebra seminorm on $A_0$ such that $\nm s_{\max}\le 1$ for all $s\in S_0$. 
If $\Nm$ is any algebra norm on $A_0$ with 
$\Nm_0\le \Nm\le \Nm_{\max}$, then writing $A$ for the completion of $(A_0,\Nm)$,  
$A$ has a cai, the spectrum of $g\in A$ is precisely $\{2^{-n}:n\in\bN\}\cup\{0\}$,
and the spectral idempotents for the eigenvalues $2^{-n}$, obtained from the analytic functional calculus
for $g$, are the unit vectors $e_n$.  Finally, $A$ is singly generated by $g$.
\elem
\proof  Expression \refeq{4-1} is at least the $c_0$-norm if and only if every $x\in S$ has $c_0$-norm at most 1, if and only if every $x\in S_0$ has $c_0$-norm at most 1. 

Looking at the definition of $S_0$, it is obvious that the vectors  $g$ and  $a_n^{-1}e_n$  and $\frac n{n+1}g_{{a}_n}^{a_n}$ have $c_0$-norm at most 1; also  the $l^\infty$-norm of $g_{{a}_n}^{a_n}-1$ is 1, but its $j$th entry is zero 
for $j\le{a}_n$, hence the $c_0$-norm of $(g_{{a}_n}^{a_n}\cdot g-g)$ is at most $\nm{g(1-P_{{a}_n})}_0=$ $2^{-(1+{a}_n)}<\frac 1n$.  So $\nm{n(g_{{a}_n}^{a_n}\cdot g-g)}_0<1$  also. 
 Finally the $c_0$-norm of $g^{a_n}(1-P_{{a}_n})$ is $2^{-a_n(1+{a}_n)}$ so since $a_n\ge n$,
the  $c_0$-norm of $n^{a_n}g^{a_n}(1-P_{{a}_n})$ is at most 1 also. Thus $\nm x_{0}\le 1$ for every $x\in S_0$, and $\nm x\ge \nm x_{0}$ for every $x\in A_0$. 

Let $\Nm$ be any algebra seminorm on $A_0$ such that $\nm s\le 1$ for all $s\in S_0$. Then plainly we have $\nm s\le 1$ for all $s\in S$. And given $\Nm$ is a norm such that $\nm s\le 1$ for all $s\in S$, plainly we must have $\nm s\le \nm s_{\max}$ as given in \refeq{4-1}.

Every element in the set $S$ has $\Nm_{\max}$ norm at most 1, so $\Nm_{\max}$ is indeed the maximal seminorm with this property, as claimed in the lemma. Also, if $\Nm$ is any algebra norm with $\Nm_0\le\Nm\le\Nm_{\max}$, we have 
\begin{equation} \label{e1}
\textstyle{\nm{g}\le 1, \; \; \; \; \;  \textrm{and} \; \;  \; \nm{e_n}\le a_n}
\end{equation}
for each $n\in\bN$; also
$$\textstyle{\nm{g_{{a}_n}^{a_n}}\le 1+\frac 1n;} \; \; \; \textstyle{\nm{g_{{a}_n}^{a_n}\cdot g-g}\le \frac 1n; \; \; \;  \textrm{and}} \; \; 
\textstyle{\nm{g^{a_n}(1-P_{{a}_n})}\le n^{-a_n}.}$$
By \reflem{1}, the vectors $x_n=\frac n{n+1}g_{{a}_n}^{a_n}$ are a cai for $A$.
By \reflem{2},  the spectrum of $g\in A$ is precisely $\{2^{-n}:n\in\bN\}\cup\{0\}$,
and the spectral idempotents for the eigenvalues $2^{-n}$, obtained from the analytic functional calculus
for $g$, are  the unit vectors $e_n$. 
By \refcor{3}, $A$ is topologically generated by $g$.
\endproof

\begin{corollary} \label{c4}
  Let $\Nm$ be an algebra norm on $A_0$ with $\Nm_0\le \Nm\le \Nm_{\max}$, and let $A$ be the completion of $(A_0,\Nm)$, and $J$ the Jacobson radical of $A$. Let $q:A\to A/J$ be the quotient map.
Then the norm $\Tnm$ on $A_0$ with $\tnm x=\nm{q(x)}$ also satisfies $\Nm_0\le \Tnm\le \Nm_{\max}$, and the conclusions of  {\rm  \reflem{4}}\ are also satisfied when $A$ is replaced by the semisimple Banach algebra $A/J$. 
\end{corollary}

\proof   Since $A$ satisfies the conditions of the previous result, and therefore 
of Corollary \ref{3} too, the characters on $A$ are the $\chi_n$ mentioned there.
Of course $J$ is the intersection of the kernels of the characters of $A$, hence
$A_0 \cap J = (0)$.  Therefore the quantity $\tnm x=\nm{q(x)}$  is a norm on $A_0$, and it is
dominated by $\Nm_{\max}$ clearly.
It is only necessary to show that $\tnm x\ge \nm x_0$; for then $A/J$ is the completion of $(A_0,\Tnm)$ and the conclusions of \reflem{4}\ will follow for the norm $\Tnm$. But 
$\tnm x$ dominates $\sup\{|\chi(x)|:\chi$ is a character of $A\}\ =\nm x_0$. \endproof

\section{A general (Banach algebraic)  theorem}

Write  $\Del_n=P_{{a}_{n+1}}-P_{{a}_n}$, and  $$H_n=\lin\{e_j:{a}_n<j\le{a}_{n+1}\}=\Del_nA_0.$$
A basis for the dual of this vector 
space is the set of characters $\{\chi_j:\range j{{a}_n+1}{{a}_{n+1}}\}$. For these $j$, we have $\chi_j=\chi_j\circ \Del_n$. 

 We have the following general theorem, which is part of what we need,  but it does not necessarily give an operator algebra, merely a Banach algebra.  However  we will use it later to get an operator algebra.

\begin{theorem} \label{sec}
 Let a strictly increasing sequence $({a}_n)$ be given, and let $A$ be the completion of $A_0$ in an algebra
norm $\Nm$, where $\Nm_0\le \Nm\le {\Nm}_{\max}$, with $\Nm_0$ being the $c_0$-norm, and $\Nm_{\max}$ the maximal norm as defined in {\rm \refeq{4-1}}. 
Suppose in addition that there is a bounded sequence $(\psi_n)_{n=1}^\infty\in A^*$ such that 
$\psi_n(g)=1$ and $\psi_n=\psi_n\circ \Del_n$ for each $n$. 
Let $B=A/\rad A$. Then $B$ has the following properties.
\begin{enumerate}
\item $B$ is a Banach algebra with cai, topologically generated by $g$.
\item $g$ is a contraction with spectrum  $\{2^{-n}:n\in\bN\}$;
\item $g\notin \overline{{\rm lin}} \, \{e_i:i\in\bN\}$ and 
\item $B$ is semisimple.
\end{enumerate}
\end{theorem}

\proof That $B$ has a cai, and the spectrum of $g$ is as stated, and that $B$ is topologically generated by $g$, follows from \refcor{c4}. Certainly $B=A/\rad A$ is semisimple, so it remains to show that 
$g\notin\overline\lin\{e_i:i\in\bN\}$. To this end, let $\psi\in A^*$ be any weak-* accumulation point of the (by hypothesis bounded) sequence $\psi_n$. Since $\psi_n=\psi_n\circ\Del_n$, we have $\psi_n\in\lin\{\chi_j:{a}_n<j\le{a}_{n+1}\}$.  Hence each $\psi_n$ annihilates $\rad A$, so $\psi$ annihilates $\rad A$, and yields a well defined element of $B^*$. Since $\psi_n(e_j)=0$ for $j\le{a}_n$ and ${a}_n\to\infty$, we have $\psi(e_j)=0$ for all $j\in\bN$. But  $\psi_n(g)=1$ for all $n$, so $\psi(g)=1$. Therefore    $g\notin\linbar\{e_i:i\in\bN\}$. 
\endproof

\section{Representations on Hilbert space} \label{hilr}

It is time to construct norms which will  presently turn $A_0$ into a nontrivial operator algebra.  These norms will be somewhat 
universal, in the sense that they are defined so as to encode  abstractly the critical hypotheses in the Lemmas in Section \ref{gco},
 which ensure that conditions (1) and (2) at the start of that Section 
hold (after which we will proceed to prove that (3) and (4) also hold).

For $k\in\bN_0$, write  $\gam_k\up n =$  $\summ j{1+{a}_n}{{a}_{n+1}}2^{-jk}e_j$.    In the rest of our paper we will very
often silently use the relation $$\gam_k\up n \gam_i\up n = \gam_{k+i}\up n .$$ The reader can check the following relations: 
We have  $\Del_ng=\gam_1\up n$, and 
$\Del_ng_{{a}_j}^{a_j}=2^{a_j^2}\gam_{a_j}\up n$ if $j\le n$, while  $\Del_ng_{{a}_j}^{a_j}= \gam_0\up n$ if $j>n$.  So
$\Del_ng_{{a}_j}^{a_j}\cdot g-g=2^{a_j^2}\gam_{1+a_j}\up n-\gam_1\up n$ if $j\le n$, and is zero if $j>n$.  Similarly, 
$\Del_ng^{a_j}(1-P_{{a}_j})=\gam_{a_j}\up n$ if $j\le n$,  and is  zero if $j>n$.
Let us also write $$\Lam_n=\{\summ i1nt_ia_i:t_i\in\bN_0,t_i\le{a}_i \; \text{for} \; \range 1in\},$$ and 
$$\xi_n=\max\Lam_n=\summ i1n{a}_i^2.$$
We shall assume that the sequence $(a_n)$ increases sufficiently fast that 
these sums are distinct for distinct sequences $(t_i)$, and that they appear in ``lexicographic order''. So, we assume that,  for $t_i\in\bZ$, $|t_i|\le 2a_i$, we have $\summ i1n t_ia_i> 0$ if and only if 
$t_r>0$, where $r=\max\{j:t_j\ne 0\}$. This condition, slightly stronger than our immediate need, will also ensure that elements of the set $\Lam_n-\Lam_n=\{\summ i1n t_ia_i:-a_i\le t_i\le a_i\}$ also appear in ``lexicographic'' order.

We note that any collection of $m$ of the vectors $\gam_k\up n$ is linearly independent, provided that $m\le{a}_{n+1}-{a}_n$.  
So a linear functional $\phi\in\lin\{\chi_j:{a}_n<j\le {a}_{n+1}\}$ is specified uniquely by its action on $\{\gam_k\up n:0\le k<{a}_{n+1}-{a}_n\}$. Let $\phi_n$ be the unique such functional such that for all $0\le k<{a}_{n+1}-{a}_n$,
\begin{equation} \label{phin}  \phi_n(\gam_k\up n)=\bcas
\prodd i1n 2^{-t_i{a}_i^2}(1-t_i/{a}_i)&\text{if } k=1+\summ i1n t_ia_i\in 1 + \Lam_n.\\
0&\text{otherwise.}
\ecas
\end{equation}
Here the $t_i \in \bN_0$ with $t_i \leq a_i$ of course, 
and $\phi_n(\gam_1\up n)=1$.   We may view $\phi_n$ as a functional on $A$ satisfying
$\phi_n=\phi_n \circ \Del_n$.  
Therefore from \refeq{phin}\ we have that
$$\phi_n(g)=\phi_n(\Del_ng)=\phi_n(\gam_1\up n)=1.$$
However, for $j\le{a}_n$ we have $\phi_n(e_j)=0$. We could go on from here and prove directly that the $\phi_n$ are uniformly $\Nm_{\max}$-bounded, whereupon we could apply Theorem \ref{sec}; but this would not get us an operator algebra. Instead we proceed as follows.

We have established that $\Del_ng=\gam_1\up n$, and 
$\Del_ng_{{a}_j}^{a_j}=2^{a_j^2}\gam_{a_j}\up n$ if $j\le n$, while $\Del_ng_{{a}_j}^{a_j}= \gam_0\up n$  if $j>n$) 
Also, $\Del_ne_j=e_j$ if ${a}_n<j\le{a}_{n+1}$, and is zero otherwise. Now $\gam_0\up  n$ is the identity of 
$\Del_nA_0$, so referring to \refeq{s0e}, and removing from $\Del_nS_0$ positive scalar multiples of the identity or of 
other elements of $\Del_nS_0$, we are left with the set
\begin{equation} \label{dns0} 
\{\gam_1\up n, a_i^{-1}e_i, \frac j{j+1}2^{a_j^2}\gam_{a_j}\up n,
j(2^{a_j^2}\gam_{1+a_j}\up n-\gam_1\up n):{a}_n<i\le{a}_{n+1}, 1\le j\le n\}.
\end{equation}
\begin{definition} \label{hn}
  The set given by {\rm \refeq{dns0}}\ will be called $S_0\up n$.  Let $\cI\up n$ denote the set of all ``index functions'' $\bi:S_0\up n\to\bN_0$. 
For $\bi\in\cI\up n$, write $\bs^\bi$ for the product $\prod_{s\in S_0\up n}s^{\bi(s)}$.
 Equip $H_n$ with a Euclidean seminorm $\Nm_2\up n$ as follows. For $x\in H_n$, we
define 
\begin{equation} \label{nm2}
{\nm x}_2\up n=(\sum_{\bi\in\cI\up n}|\phi_n(\bs^\bi x)|^2)^{1/2}.
\end{equation} \end{definition}  

We shall establish that ${\nm x}_2\up n$ is finite,  so that we do indeed have a Euclidean seminorm.   Letting ${\mathcal
H}_n$ denote 
the associated Euclidean  space, we shall represent each $T\in A_0$ in $B({\mathcal
H}_n,\Nm_2\up n)$ by its compression
$\Del_nT$.  This representation will be called $\rho_n:A_0\to B({\mathcal
H}_n,\Nm_2\up n)$, and the operator norm
$\nm{\rho_n(T)}$ will be called $\opnmn T$.  
  We then define on $A_0$ the quantity
\begin{equation} \label{opnm}
  \nm T=\sup_{n \in \bN_0} \;  \opnmn T,
\end{equation}
where $\Vert T \Vert_{\rm op}^{(0)} = \nm T_{0}$.   We shall show that this is a norm.

The basic Lemma we shall prove is as follows:

\blem\l nm21! 
For every $x\in H_n$, we have ${\nm x}_2\up n<\infty$, provided that our underlying sequences satisfy certain growth conditions.
The representation $\rho_n$ above on ${\mathcal
H}_n$ is well defined, and the operator norm $\Vert T \Vert^{(n)}_{{\rm op}}
\le 1$ for every $T\in S_0$.
\elem

\proof   
 Consider $H_n$  as an algebra with pointwise product.  The spectral radius here is the $c_0$-norm, and for the various elements of $S_0\up n$ it is as follows: $$\nm{\gam_1\up n}_0=2^{-1-{a}_n} \; , \; \; \;  \nm{a_i^{-1}e_i}_0=a_i^{-1}, \; \; \; \nm{\frac j{j+1}2^{a_j^2}\gam_{a_j}\up n}_0=
\frac j{j+1}2^{a_j({a}_j-{a}_n-1)},$$
and $$\nm{j(2^{a_j^2}\gam_{1+a_j}\up n-\gam_1\up n)}_0
\le j(2^{a_j({a}_j-{a}_n)}+2^{-1-{a}_n}).$$ A mild growth condition on the $(a_j)$  will ensure that for all $n$,
$$2^{-1-{a}_n}+\summ i{{a}_n+1}{{a}_{n+1}}a_i^{-1}+\summ j1n (j+1)(2^{a_j({a}_j-{a}_n)}+2^{-1-{a}_n})<1.$$
In particular,  the spectral radii of the elements of $S_0\up n$ are then all strictly less than 1.
 It is a nice exercise that on any commutative algebra with a finite set of generators of spectral radius
$<1$, we can pick an algebra norm such
that $||s||<1$  for all generators $s$
simultaneously (Hint: let $G_0$ be the set of generators, each multiplied by $1+ \epsilon$, such that the semigroup $G$ generated by $G_0$
is bounded.   Renorm $A$ in the usual way so that $G \subset {\rm Ball}(A)$,
see e.g.\ \cite[Proposition 1.1.9]{Pal}). With respect to such a norm on $H_n$,
the square of the sum in  \refeq{nm2}\ is at most
$$\sum_{\bi} \prod_{s\in S_0 \up n} \, || s ||^{2 \bi(s)} \cdot ||\phi_n||^2   =
||\phi_n||^2 \prod_s (1-||s||^2)^{-1} < \infty.$$
So the expression ${\nm x}_2\up n$ in \refeq{nm2}\ is finite.

It is clear from \refeq{nm2}\ that for $s\in S_0\up n$ and $x\in H_n$,
the terms in the sum defining  ${\nm {x}}_2\up n$ include all those in the sum
defining  ${\nm {sx}}_2\up n$ (plus certain extra terms, namely $|\phi_n(\bs^\bi x)|^2$ for index functions $\bi$ such that $\bi(s)=0$).
Therefore ${\nm {sx}}_2\up n\le  {\nm {x}}_2\up n$, and
$\opnmn s\le 1$ for $s\in S_0\up n$. And for $s\in S_0$, we have
$\opnmn s=\opnmn{\Del_ns}$, which is either a multiple of the identity $\gam_0\up n$ of magnitude
less than 1, or an element of $S_0\up n$ (again possibly  multiplied by a positive scalar of magnitude
less than 1).
Thus $\opnmn s\le 1$ for all $s\in S$,
 as required.  The last few lines also show that $\rho_n$ above is well defined.
 \endproof

\medskip

It will be much harder to prove the next result.  In fact
this proof will take up almost all of the rest
of our paper. 

\blem\l nm22!
Given growth conditions on our underlying sequences, we have 
$0<{\nm {\gam_0\up n}}_2\up n\le 3\cdot {\nm {\gam_1\up n}}_2\up n$ for all $n$.
We have  $\Vert g \Vert^{(n)}_{{\rm op}} \ge \frac 13$.
\elem

Taking this lemma on faith for now, the rest of our 
assertions follow rather easily:

\begin{theorem} \label{sac} 
Given growth conditions  on the $a_n$, the operator algebra norm defined in {\rm \refeq{opnm}}
 is at most $\Nm_{\max}$. The completion $A$ of $A_0$ in this norm is an operator algebra satisfying all of the conditions {\rm (1)--(4)} at the start of Section  {\rm  \ref{gco}}. \end{theorem}

\proof   
The completion $A$ of $A_0$ in the norm $\Nm$ defined in {\rm \refeq{opnm}} will be an operator algebra whose norm lies between $\Nm_0$ and $\Nm_{\max}$ as in \reflem{4} (for $\Nm_{\max}$ is the largest norm on $A_0$ such that ${\nm s}_{\max}\le 1$ for all $s\in S_0$). Furthermore, since $\opnmn g\ge \frac 13$
for each $n$ by Lemma \ref{nm22}, there is a linear functional $\psi_n\in(A_0,\Opnmn)^*$
such that $\nm{\psi_n}\le 3$ and $\psi_n(g)=1$.
Finally, $A$ is semisimple (there is no need to quotient out by $\rad A$), because
the operator norm $\opnmn T$ is zero unless one of the characters $\chi_j(T)\ne 0$ for some $j$ with
$\alp_n< j\leq \alp_{n+1}$.
Thus  $(A,\Nm)$ satisfies the requirements of  \refthm{sec}: but it is clearly an operator algebra.
 \endproof

\begin{corollary} \label{maxse}
  If  $A_{\rm max}$ is the completion
of $A$ in the maximal norm defined in {\rm \refeq{4-1}},
and given growth conditions on the $a_n$,  then $A_{\rm max}$ (resp.\ $A_{\rm max}/ \rad A_{\rm max}$) is a Banach algebra satisfying the desired
conditions  {\rm (1)--(3)} (resp.\  {\rm  (1)--(4)}) at the start of Section  {\rm \ref{gco}}.  \end{corollary} 

\proof  By Lemma \ref{4} and Corollary \ref{c4}, $A_{\rm max}$ (resp.\ $A_{\rm max}/ \rad A_{\rm max}$) 
satisfy (1)--(2) (resp.\ (1)--(2) and (4)).  If $\psi_n$ is as in the proof of
Theorem \ref{sac}, and if $i : A_{\rm max} \to A$ is the canonical contraction
due to $\Nm_{\max}$ being a larger norm, then $\psi_n \circ i$ satisfies the conditions
of Theorem \ref{sec} with $A$ replaced by $A_{\rm max}$.  Then Theorem \ref{sec} 
or its proof implies that (3) also holds.  \endproof

\begin{remark}
In a previous draft Corollary \ref{maxse}  was proved directly, but for reasons of space this
(very lengthy) computation has been omitted.
\end{remark}

\section{Additional properties of our algebras $A$ and $\overline{A_{00}}$}  \label{Addpro}

In this section we take Lemma \ref{nm22} on faith, so that our main theorems above hold for our algebra $A$.

Let $H = \ell^2 \oplus (\bigoplus^2_{n\in \Ndb} \, {\mathcal
H}_n)$, and 
write $\rho$ for the representation of $A$ that we have already 
constructed.  Namely
$$\rho(a) = \Gamma(a) \oplus (\bigoplus_{n\in \Ndb} \, \rho_n(a)) .$$
Here $\Gamma$ is the Gelfand transform, mapping into $c_0$, but with 
$c_0$ viewed 
as `diagonal' operators on $\ell^2$ in the usual way.
Note that $\rho$ is initially defined on $A_0$, and is 
easily seen to be a nondegenerate
representation of $A_0$.  The norm on $A$ was defined in such a way that 
$\rho$ extends to a (completely) isometric representation of $A$, which we will
continue to write as $\rho$, and which is still nondegenerate.

\blem\l noca!  The algebra $\overline{A_{00}}$ has a contractive 
 approximate identity.   Also $A$ is a subalgebra of the closure of
$\overline{A_{00}}$
 in the strong operator 
topology of $B(H)$ for $H$ as above.  \elem
\proof  Indeed $(\rho(P_{a_{k+1}}))$ is a cai for $A_{00}$,
and hence also for $\overline{A_{00}}$, since $\rho_n(P_{a_{k+1}})$ is just $I_{{\mathcal
H}_k}$ if
$k \geq n$, and is zero for $k < n$.  

Clearly $\rho(P_{a_{k+1}}) \to I$ strongly on $H$, and
hence $\rho(g P_{a_{k+1}}) = \rho(g) \rho(P_{a_{k+1}})  \to \rho(g)$ strongly.
\endproof

\medskip

Hence $\overline{A_{00}}$ is a (complete) $M$-ideal in its bidual by \cite[Theorem 4.8.5]{BLM}, with all the consequences that this
brings (see e.g.\ \cite[Chapter 3]{HWW} and \cite[Theorem 5.10]{ABR}).  
Note that the explicit representation 
$\rho$ of $\overline{A_{00}}$ given above
 shows that  $\overline{A_{00}}$ is a subalgebra of the compact operators on $H$
(since each $\rho(e_n)$ is compact on $H$).  
The characters on $\overline{A_{00}}$ are again the $\chi_n$ above of course, 
since for any such character $\chi$ we must have $\chi(e_n) \neq 0$ for some $n$,
and  then 
$\chi(e_m) = 0$ for all other $m \in \Ndb$, so that $\chi = \chi_n$.  The spectrum of $\overline{A_{00}}$ is thus the same as 
the spectrum of $A$ (and equals the spectrum of $\overline{A_{00}}^{**}$ by a point above).  
Note that the $\rho(P_{a_{k+1}})$ above is a (contractive) spectral resolution of the identity.

\begin{corollary} \label{finlyw}  The operator algebra $A$ constructed in Section {\rm 5} 
is not compact or weakly compact.  Thus 
$A$ is not an ideal in its bidual.
\end{corollary}   
\proof  See e.g.\ \cite[1.4.13]{Pal} or 
\cite[Lemma 5.1]{ABR} for the well known
equivalence between being weakly compact
and being an ideal in the bidual. 

The rest follows from results at the end of the Introduction,
but we give a direct proof.  Assume by way of contradiction that multiplication by $g$ on  $A$ (or, for that matter, on  $\overline{A_{00}}$) is  weakly compact.
Suppose that $(f_k)$ is a bai for $\overline{A_{00}}$.  
Since $\overline{A_{00}}$ is weakly closed,
there is a subsequence
$g f_{n_k} \to a \in \overline{A_{00}}$ weakly, say.  Thus $g f_{n_k} e_m \to a e_m$ in norm for each $m \in \Ndb$.
However $g f_{n_k} e_m \to g e_m$ in norm since $(f_k)$ is a bai.  Thus $(g - a) e_m = 0$ for every $m$, yielding the
contradiction $g=  a \in \overline{A_{00}}$.  
  \endproof

\medskip

The proof of Theorem \ref{ch} is now complete, except for Lemma \ref{nm22} and the
very last assertion, about obtaining $g \notin  \overline{A_{00}}^{{\rm SOT}}$.  To see the latter, 
we will change the Hilbert space $A$ acts on.  Suppose that $A$ generates a $C^*$-algebra $B$.
Then  $\overline{A_{00}}$ generates a proper $C^*$-subalgebra $B_0$ of $B$ (since if $B_0 = B$
then the cai of $\overline{A_{00}}$ would be a cai for $A$ by \cite[Lemma 2.1.7 (2)]{BLM}, and this is false).  Let $B \subset B(K)$ be the 
universal representation, so that $B^{**}$ may be represented
as a von Neumann algebra $M$ on $K$.  Then $M = \overline{B}^{w*} = \overline{B}^{{\rm SOT}}$
by von Neumann's double commutant theorem.   If $A \subset \overline{A_{00}}^{{\rm SOT}}$
then $$A
\subset \overline{B_0}^{{\rm SOT}} = \overline{B_0}^{w*},$$
so that  $B^{**} \cong \overline{B}^{w*} \subset \overline{B_0}^{w*}$.  This implies that $B_0^{\perp \perp} = B^{**}$,
and   we obtain the contradiction $B = B_0$.

\begin{remark}
 Probably a modification of our construction in Section  \ref{hilr} would produce a representation 
in which we would explicitly have $g \notin  \overline{A_{00}}^{{\rm SOT}}$.  We had a more complicated
construction for Theorem \ref{ch} 
in a  previous draft for which this  perhaps may have been true.  
\end{remark}

We recall that for a commutative semisimple Banach algebra
the following are equivalent: (i)\  $A$ is a modular annihilator algebra; 
(ii)\ the Gelfand spectrum of $A$ is discrete; (iii)\  no element of $A$ has a nonzero
limit point in its spectrum; and (iv)\ for 
every $a \in A$, multiplication on $A$ by $a$
is a Riesz operator (see \cite[Theorem 8.6.4 and Proposition
8.7.8]{Pal} and \cite[p.\ 400]{LN}).  Thus $A$ and $\overline{A_{00}}$ are
modular annihilator algebras.   It  follows that our algebra $A$ is
a commutative solution to a problem raised in \cite{BRII}: is
 every semisimple modular annihilator algebra with a cai, weakly compact?  
In \cite{BRIII} we found a
much simpler (but still deep) 
noncommutative counterexample to the latter question, an example with some interesting noncommutative features.  

\begin{theorem} \label{addpro}   The operator algebras $A$ and $\overline{A_{00}}$ constructed above have the following additional properties: \begin{itemize}
\item [{\rm (a)}]  Every maximal ideal in $\overline{A_{00}}$, and every maximal modular ideal in $A$,  has a bounded approximate identity.
\item [{\rm (b)}]  $A$ and $\overline{A_{00}}$ are regular natural Banach function algebra (in the sense of 
\cite[Section 4.1]{Dales}) on $\bN$ or on $\{ \frac{1}{2^n} : n \in \bN \}$. 
\item [{\rm (c)}]  $A$ is not Tauberian, nor is strongly regular or Ditkin, nor satisfies spectral synthesis (see \cite{Dales,LN})
for definitions).  On the other hand, $\overline{A_{00}}$ does have all these properties, indeed it is a strong Ditkin algebra. 
\item [{\rm (d)}]   $A$ is a semisimple modular annihilator algebra, while $\overline{A_{00}}$ is 
a dual algebra in the sense of Kaplansky (see e.g.\ \cite[Chapter 8]{Pal}).
\item [{\rm (e)}]  The closure of the socle of $A$ (or of $\overline{A_{00}}$)
is $\overline{A_{00}}$, and $A$ is not an annihilator algebra in the sense of 
\cite[Chapter 8]{Pal}.
\item [{\rm (f)}]   $A$ is not nc-discrete in the sense of \cite{ABS}: indeed the support projection of $\overline{A_{00}}$ 
in $A^{**}$ is open but not closed.
\item [{\rm (g)}]   $A$, and its multiplier algebra $M(A)$, may be identifed completely isometrically isomorphically with subalgebras of the 
multiplier algebra $M(\overline{A_{00}})$.
\end{itemize} 
\end{theorem}
 
\proof   We only prove some of these assertions, leaving the others as exercises.

(a)\ The maximal ideals are the annihilators of the $e_m$, which have as a bai  $(x_n - x_n e_m)$, where $(x_n)$ is a cai for $A$ or $A_0$. 

(b)\ These follow easily from the definitions, and the identification of the characters of these algebras.

(c)\  First, $A$ is not Tauberian
in the sense of e.g. \cite[Definition 4.7.9]{LN}), because $\overline{A_{00}} \neq A$.  This implies 
failure of spectral synthesis by e.g.\ \cite[p.\ 385]{LN}.  Similar arguments show the other assertions
for $A$.   The statements for $\overline{A_{00}}$ are easy, or follow
from \cite[p.\ 419]{Dales}.  

(d)\ We have already observed this 
for $A$.  For $\overline{A_{00}}$ this follows from \cite[Proposition 4.1.35]{Dales}, or from the observation
whose proof we omit
that for a natural Banach sequence algebra being `dual' is equivalent to spectral synthesis holding, or to having 
`approximate units' \cite[Definition 2.9.10]{Dales}.   

(e)\ The assertion for $\overline{A_{00}}$ is clear.  If
$f$ is a minimal idempotent in $A \setminus A_{00}$, then $f e_n = 0$ for all
$n \in \Ndb$ (since  $f e_n \in \Cdb f \cap \Cdb e = (0)$), and so $f = 0$.

(f)\ This is almost identical to the proof of \cite[Corollary 2.13]{BRIII}, except that we work with 
$e_n$ as opposed to the $e^n_{ii}$ there.  Note that $\tilde{\rho}(1-p) e_n \neq 0$ for some $n$
because otherwise the (faithful) Gelfand transform of $\tilde{\rho}(1-p) a$ would be zero for all
$a \in A$.  

(g)\  In the explicit nondegenerate representation 
$\rho$ of $A$  given in 
Section 5 and the start of Section 6, 
it is clear that $\rho(A) \rho(A_{00}) \subset \rho(\overline{A_{00}})$.  More generally, if $T \rho(A) 
\subset \rho(A)$ then $$T \rho(\overline{A_{00}}) = T \rho(\overline{A_{00}}) \rho(\overline{A_{00}})
\subset \rho(A) \rho(\overline{A_{00}}) \subset \rho(\overline{A_{00}})  .$$
A similar assertion holds if $\rho(A) T 
\subset \rho(A)$.  The results now follow
from  basic facts about multiplier algebras \cite[Section 2.5]{BLM}.
\endproof

\begin{remark}
  1)\ 
By  \cite[Theorem 5.10]{ABR}, 
the dual space of $\overline{A_{00}}$ has no proper closed subspace
that norms $\overline{A_{00}}$.   On the other hand,  
one can show that the closure of the span of the characters of $A$ 
in $A^*$,  is a proper norming subspace for $A$.

2)\  For our algebra, the interested reader can easily identify the algebras $M_0(\overline{A_{00}})$ and $M_{00}(\overline{A_{00}})$ studied 
in \cite{LN}.

\smallskip

3)\   It is easy to see that $g^n$ generates $A$ for every $n \in \Ndb$. 
Also, $A/\overline{A_{00}}$ is an interesting commutative radical operator algebra with cai.  
\end{remark}

Finally we give the illustration  promised after Theorem \ref{ch}, of how our examples can be useful in settling open questions in the immediate area.  The unitization $A^1$ of our algebra $A$ 
 solves an old question of Joel Feinstein  (see e.g.\ \cite{JF1,JF2}), and another similar question of Dales. 
They asked (in language explained for example
 in  \cite[Section 4.1]{Dales}): If $A$ is a regular unital Banach function algebra on its
 compact character space $X$, and if $E$ is a closed subset
of $X$ such that the the ideal $M_E$ of functions
 in $A$ vanishing on $E$ has a bai (for regular uniform algebras
 this is equivalent to $E$ being a $p$-set or generalized
 peak set),  then is $E$ a set of
 synthesis?  That is, is the ideal $J_E$ of functions in $A$ with a compact support which is  disjoint from $E$,
dense in $M_E$?     Feinstein's question was the case that
 $E$ is a singleton $\{x \}$; in this case it is
 equivalent to ask if
 $A$ is `strongly regular' at $x$.  (If $A$ is a uniform
 algebra then this is related to an even older important open problem, which would have
 remarkable consequences if true.)   Feinstein also asked if $M_x$ has a bai for every $x \in X$ then is $A$  `strongly 
regular' (that is every point in $X$  a set of synthesis)?
 
 As we have said in Theorem \ref{addpro}, our algebra $A$ is a separable
regular Banach function algebra, and by facts in \cite[Section 4.1]{Dales}
 we have that the unitization  $A^1$ is also a regular unital Banach function algebra on the one-point compactification
 $\{ 0 \} \cup \{ \frac{1}{2^n} : n \in \Ndb \}$.   It is clear that if $E$ is the single adjoined point then $M_E = A$ has a
 cai, but $J_E = A_{00}$
 is not dense in $M_0 = A$, answering the questions in the last paragraph
 in the negative.

\section{A lower estimate on $\nm{\gam_1\up n}_2\up n$}

We now return to the task of proving Lemma \ref{nm22}.

\blem\l le!
Given growth conditions on our underlying sequences, the following is true: for all $n\in\bN$, we have 
$(\nmtn{\gam_1\up n})^2\ge \frac 12\cdot\prodd j1n\frac{(j+1)^2}{2j+1}$\ \ . 
\elem
\proof Consider index functions $\bi\in\cI\up n$ such that $\bi(s)=0$ unless $s=s_j=\frac j{j+1}2^{a_j^2}\gam_{a_j}\up n$ for some $\range j1n$; and for each such $s_j$, we have 
$\bi(s_j)<{a}_j$. For such an $\bi$, the product $\bs^\bi\gam_1\up n$ is given by the formula  
$$
\bs^\bi\gam_1\up n=\prodd j1n(\frac j{j+1}2^{a_j^2})^{i_j}\gam_{1+\summ j1na_ji_j}\up n,
$$
where $i_j=\bi(s_j)$. The value of $\phi_n(\gam_{1+\summ j1na_ji_j}\up n)$ is given by \refeq{phin}, and its value is
$$\prodd j1n\linebreak 2^{-a_j^2i_j} (1-i_j/{a}_j).$$ So the sum, for these $\bi\in\cI\up n$, of $|\phi_n(\bs^\bi\gam_1\up n)|^2$, is precisely
\begin{equation} \label{nnr} \sum_{{i_j=0, ..., {a}_j-1}\atop{\range j1n}} (\prodd j1n(\frac j{j+1})^{i_j}(1-i_j/{a}_j))^2
= \prodd j1n \sum_{i=0, ..., {a}_j-1} \, (\frac j{j+1})^{2i} \, (1-i/{a}_j)^2.
\end{equation}
We may assume that ${a}_j$ is very large indeed compared to $j$. We claim that growth conditions will ensure that  this quantity 
will be, for each $n$, at least half of the sum
$$
\prodd j1n \sum_{{i \in \bN_0}\atop{\range j1n}} (\frac j{j+1})^{2i}
=  \prodd j1n\frac 1{1-(j/j+1)^2}=\prodd j1n\frac{(j+1)^2}{2j+1} . 
$$
To see this, let us choose positive $h_k<1$ such that $\prod_{k= 1}^m  h_k > 1/2$ for all $m \in \Ndb$.  Our product
\refeq{nnr}  will be at least half of the product of $(j+1)^2/(2j+1)$, provided
that for every $j=1, \cdots, n$, we have
$$ \sum_{i=0}^{a_j-1} (\frac j {j+1})^{2i} 
(1- \frac i {a_j})^2\ge h_j \cdot \sum_{i=0}^\infty (\frac j {j+1})^{2i}.
$$
The sum 
$\sum_i (j/j+1)^{2i}$
converges, and we may think of it as a integral with respect to 
counting measure.  If  $f_a(i)=1-a/i$ for $i<a$, and is zero for $i\ge a$, then $(f_a)$ is  
uniformly bounded and converges to  1 pointwise. Therefore, by the Lebesgue
dominated convergence theorem, $$\lim_{a \to \infty} \, \sum_i (j/j+1)^{2i} f_a(i)^2 = \sum_i (j/j+1)^{2i}.$$ If each $a_j$ is chosen large enough we therefore have for $j=1, \cdots, n$ that
 $$\sum_i (j/j+1)^{2i} f_{a_j}(i)^2 =\sum_i (j/j+1)^{2i}  (1- \frac i {a_j})^2
\ge (1-h_j) \sum_i (j/j+1)^{2i} .$$
This proves the claim.
Finally, $(\nmtn{\gam_1\up n})^2\ge \frac 12\cdot \prodd j1n\frac{(j+1)^2}{2j+1}$, as desired.  \endproof

\section{Strategy for an upper estimate for $\nmtn{\gam_0\up n}$}

In the remainder of our paper we strive for an upper estimate for $\nmtn{\gam_0\up n}$. Now
\begin{equation} \label{g01}
  (\nmtn{\gam_0\up n})^2=(\nmtn{\gam_1\up n})^2+\sum_{\bi\in \cI_0\up n}|\phi_n(\bs^\bi)|^2,
\end{equation}
where $\cI_0\up n=\{\bi\in \cI\up n:\bi(\gam_1\up n)=0\}$.
Let us write $$\cI_0\up n=\cI_1\up n\cup\cI_2\up n\cup\cI_3\up n,$$
with $$\cI_1\up n=\{\bi\in\cI\up n:\bi(\gam_1\up n)=\bi(a_i^{-1}e_i)=0 \; \text{for all} \; i, \; \text{and} \; 
|\bi|=\sum_s \bi(s)<\sqrt{{a}_{n+1}}\},$$ and 
$$\cI_2\up n=\{\bi\in\cI\up n:\bi(\gam_1\up n)=\bi(a_i^{-1}e_i)=0 \; \text{for all} \; i , \; \text{but} \; |\bi|\ge\sqrt{{a}_{n+1}}\},$$
and 
$$\cI_3\up n=\{\bi\in\cI\up n:\bi(\gam_1\up n)=0, \; \text{but} \; \bi(a_i^{-1}e_i)>0 \; \text{for some} \; i\}.$$

The main contribution towards the sum \refeq{g01}\ that we must investigate, is from the sum over $\bi\in \cI_1\up n$.  We will estimate this in the lengthy Section \ref{sixse}.  In the much easier
Sections \ref{7se} and \ref{8se} we estimate the contribution from $\cI_2\up n$
and $\cI_3\up n$ respectively, and in the final Section \ref{9se} we summarize 
why this proves our main result. 
 
\section{Bound on $\sum_{\bi\in \cI_1\up n}|\phi_n(\bs^\bi)|^2$} \label{sixse} 
Let $\bi\in\cI_1\up n$. Write $$E(\bi)= \{j\in[1,n]:\bi(\frac j{j+1}2^{a_j^2}\gam_{a_j}\up n)>{a}_j\}.$$
Let $\eta(\bi)$ be the set whose elements are of form 
$$\sum_{j\in E(\bi)}(\lam_ja_j)+\summ j1n(\mu_j+\nu_ja_j),$$
for integers $$
\lam_j=\bi(\frac j{j+1}2^{a_j^2}\gam_{a_j}\up n), \; \; \; 0\le\nu_j\le\mu_j=\bi(j(2^{a_j^2}\gam_{1+a_j}\up n-\gam_1\up n)),\nu_j\in\bN_0.$$

\blem\l 6.1!
Let $\bi\in\cI_1\up n$. If $\phi_n(\bs^\bi)\ne 0$, then the set $\eta(\bi)$ defined above must contain a positive element of the set $1+\Lam_n-\Lam_n$. 
\elem
\proof For when $0\le k<{a}_{n+1}-{a}_n$, we have $\phi_n(\gam_k\up n)=0$ unless 
$k\in 1+\Lam_n$ by \refeq{phin}\ (of course things are more complicated for larger $k$). Writing $\lam_j=\bi(\frac j{j+1}2^{a_j^2}\gam_{a_j}\up n)$, the product $\prod_{j \notin E(\bi)}(\frac j{j+1}2^{a_j^2}\gam_{a_j}\up n))^{\lam_j}$ is a multiple
$\lam\cdot\gam_k\up n$, where $k=\sum_{j\notin E(\bi)}\lam_ja_j\in\Lam_n$. The full product $\bs^\bi$ is equal to 
$$\bs^\bi = \lam\cdot\gam_k\up n\cdot\prod_{j\in E(\bi)}(\frac j{j+1}2^{a_j^2}\gam_{a_j}\up n)^{\lam_j}\cdot\prodd j1n (j(2^{a_j^2}\gam_{1+a_j}\up n-\gam_1\up n))^{\mu_j},$$  where again,
$\mu_j=\bi(j(2^{a_j^2}\gam_{1+a_j}\up n-\gam_1\up n))$. This is a linear combination of vectors $\gam_m\up n$ for $m\in k+\eta(\bi)$.  Furthermore, since $|\bi|<\sqrt{{a}_{n+1}}$, given a growth condition asserting  ${a}_{n+1}$ large compared to ${a}_n$, there is never any vector $\gam_m\up n$ involved when $m\ge{a}_{n+1}-{a}_n$. So the set $k+\eta(\bi)$ must meet the set $1+\Lam_n$, hence the result.   
\endproof

\medskip

Having got \reflem{6.1}, we want to separate out the cases when $1\in\eta(\bi)$, from the cases when $\eta(\bi)$ only contains larger elements of the set $$1+\Lam_n-\Lam_n=\{1+\summ i1n t_ia_i:-{a}_i\le t_i\le{a}_i\}.$$

Let us write $m_0(\bi)=\min(\eta(\bi)\cap(1+\Lam_n-\Lam_n))$,
and let us begin with the more challenging case when $m_0=m_0(\bi)>1$.
We can then write $m_0=1+\summ i1r a_it_i$ with $t_r>0$ and $-{a}_j\le t_j\le {a}_j$ for all $j$.
In particular, $m_0\le 1+\xi_r$.
With $\lam_j$ and $\mu_j$ as above, we can write
$$m_0=\sum_{j\in E(\bi)}a_j\lam_j+\summ j1n(\mu_j+a_j\nu_j)$$ with $0\le\nu_j\le\mu_j$.
But then, we must have $\nu_j=0$ for $j>r$ otherwise the value of $m_0$ will be too big. Indeed $\nu_r=0$ too, or we can get a smaller element of $\eta(\bi)\cap(1+\Lam_n-\Lam_n)$ by considering $m_0-a_r$. Again, we cannot have $j\in E(\bi)$ for any $j\ge r$ otherwise the value of $m_0$ is again too big (these are $j$ such that $\lam_j>{a}_j$). So, $E(\bi)\subset [1,r)$ and
\begin{equation} \label{m0}  m_0=\sum_{j\in E(\bi)\subset[1,r)}\lam_ja_j+\summ j1{r-1}(\mu_j+\nu_ja_j)+\summ jrn\mu_j.
\end{equation} Let us consider the vector
\begin{equation} 
\label{x}  x= (\gam_1\up n)^{\summ jrn \, \mu_j}\cdot \prodd j1{r-1}(\frac j{j+1}2^{a_j^2}\gam_{a_j}\up n)^{\lam_j}\cdot (j(2^{a_j^2}\gam_{1+a_j}\up n-\gam_1\up n))^{\mu_j} .
 \end{equation}
Given a growth condition, we can certainly assume that $m_0\in[(t_r-\frac 14)a_r,(t_r+\frac 14)a_r]$.
For which $i\in[0,{a}_{n+1}-{a}_n)$ does $x$ have a nonzero coefficient for  $\gam_i\up n$? 
(We may refer to the set of such $i$ as ``the $\gam$-support of $x$'').
Where does the $\gam$-support of $x$ lie? 
From \refeq{x}, the generic element of that support is a sum 
$$m=\summ r1n (\mu_j+\nu_j'a_j)+\summ j1{r-1} \lam_ja_j, \qquad  0\le\nu_j'\le\mu_j
.$$ 
The difference between $m_0$ (as given by \refeq{m0}) and this expression
is a sum of form  $$\summ j1{r-1}(\nu_j'-\nu_j)a_j , \qquad 0\le\nu_j'\le\mu_j,$$
 plus
the sum $\sum_{j\in[1,r)\setminus E(\bi)} \lam_ja_j$. The second sum cannot
be negative, but in the worst case might be as large as  $\xi_{r-1}$.
Write $m_0'=\summ r1n (\mu_j+\nu_j'a_j)+\sum_{j\in E(\bi)} \lam_ja_j$.
The ratio $m_0'/m_0$ is in $[1/(1+a_{r-1}), 1+a_{r-1}]$, and we have $m_0'\le
m\le m_0'+\xi_{r-1}$. So the  $\gam$-support of $x$ is contained in
$$[m_0/(1+a_{r-1}),m_0(1+a_{r-1})+\xi_{r-1})\subset(a_r/2a_{r-1},2a_r^2a_{r-1}),$$
given a growth condition.

Let us write 
\begin{equation} \label{tau} 
\tau=\summ i1{r-1}\lam_ia_i+\summ i1n\mu_i ,  \end{equation}
noting that $\tau$ is the minimum of the $\gam$-support of $x$.

Given that the vector $x$ is $\gam$-supported well to the right of zero, let us introduce a Banach algebra norm $\Nm_\gam$ to make use of this. The $l_1$ version of this is
$$\nm{\summ i0{{a}_{n+1}-{a}_n-1}y_i\gam_i\up n}_\gam=\summ i0{{a}_{n+1}-{a}_n-1}|y_i|.$$
We have
$$\nm x_\gam\le\prodd j1{r-1}(\frac j{j+1}2^{a_j^2})^{\lam_j}(j(2^{a_j^2}+1))^{\mu_j}. $$
For a reasonable bound on this, let us write $\summ j1{r-1}\lam_j+\mu_j=L$; so the sum of all the indices $\lam_j,\mu_j$ involved in the last product 
is $L$. Since the largest possible power is $(r-1)(2^{{a}_{r-1}^2}+1)$, and 
$\Vert \cdot \Vert_{\gamma}$ is a Banach algebra norm, we get
\begin{equation} \label{nmxg}
  \nm{x}_\gam\le ((r-1)(2^{{a}_{r-1}^2}+1))^{L}.
\end{equation}
The vector $\bs^\bi$ is equal to
$$
\prodd j1{n}(\frac j{j+1}2^{a_j^2}\gam_{a_j}\up n)^{\lam_j}(j(2^{a_j^2}\gam_{1+a_j}\up n-\gam_1\up n))^{\mu_j}
=
x\cdot\prodd jrn(\frac j{j+1}2^{a_j^2}\gam_{a_j}\up n)^{\lam_j}(j(2^{a_j^2}\gam_{a_j}\up n-\gam_0\up n))^{\mu_j}
$$
 \begin{equation} \label{sixb}  
=x'\cdot\prodd j{r+1}n(\frac j{j+1}2^{a_j^2}\gam_{a_j}\up n)^{\lam_j}(j(2^{a_j^2}\gam_{a_j}\up n-\gam_0\up n))^{\mu_j},
\end{equation}
where 
\begin{equation} \label{x'}
  x'=x\cdot (\frac r{r+1}2^{a_r^2}\gam_{a_r}\up n)^{\lam_r}(r(2^{a_r^2}\gam_{a_r}\up n-\gam_0\up n))^{\mu_r}.
\end{equation}

Now $\summ j1n\mu_j\le m_0\le 1+\xi_r$ from \refeq{m0} and the definitions of $r$ and $\xi_r$.  So for each $\range j{r+1}n$ we have $\mu_j\le 1+\xi_r$ and $\lam_j+\mu_j\le a_j+1+\xi_r<2a_j$, since $j \notin E(\bi)$. When $j=r$ we have $\lam_j\le a_j$, again  because 
$r\notin E(\bi)$; but we must be content with the estimate $\mu_r\le 1+\xi_r$ above.  So $\lam_r+\mu_r\le 1+a_r+\xi_r\le 2a_r^2$, given a growth condition.

\blem\l i1e!
Given growth conditions, the following is true. For $r$ as above, and any nonnegative integer $\lam\le 2{a}_r^2$, and any 
$$x\in\lin\{\gam_j\up n:\lam a_r+a_r/2a_{r-1}<j\le (\lam+1) a_r+a_r/2a_{r-1}\},$$  we have 
\begin{equation} \label{20}
  |\phi_n(x)|\le 2^{-(\lam+1)a_r^2}{\nm x}_\gam .
\end{equation}
Furthermore, if $\lam> 1+{a}_r$, then $\phi_n(x)=0$.
\elem

\proof Equation \refeq{20}\ is equivalent to $$|\phi_n(\gam_{i}\up n)|\le 2^{-(1+\lam)a_r^2}, \qquad i\in(\lam a_r +a_r/2a_{r-1},(\lam+1) a_r +a_r/2a_{r-1}].$$ By \refeq{phin}, the left hand side is zero unless $i\in 1+\Lam_n$. But, given a growth condition, we can assume 
$a_r/2a_{r-1}>\xi_{r-1}=\max\Lam_{r-1}$, so the element of $1+\Lam_n$ involved must be at least 
$1+(\lam+1)a_r$.  Equation \refeq{phin}\ then gives 
$|\phi_n(\gam_{i} \up n)|\le 2^{-(1+\lam)a_r^2}$
as required. If $\lam>1+{a}_r$ we have $i>1+\xi_r$, so the least element of $1+\Lam_n$ available would be $1+a_{r+1}$. But given a growth condition, we can certainly assume that the absolute upper bound
$i\le (2{a}_r^2+1)a_r + a_r/2a_{r-1}$ is less than $a_{r+1}$, so for $\lam>1+{a}_r$ we have 
$\phi_n(x)=0$.\endproof

\medskip

We now use the lemma to estimate $|\phi_n(x')|$, where $x'$ is as in \refeq{x'}.  Define $t$ to be the nonnegative integer with 
\begin{equation} \label{deft}
\tau\in[ta_r,(t+1)a_r),
\end{equation} 
 where $\tau$ is the minimum of the support of $x$ as in \refeq{tau}\ and \refeq{x}. We will have $0\le t\le a_r$ because from \refeq{m0}, $$\tau\le m_0+\sum_{i\in [1,r)\setminus E(\bi)}\lam_ra_r\le m_0+\xi_{r-1},$$ and then $$m_0\le 1+\xi_r=1+\xi_{r-1}+a_r^2,$$ so $\tau \le 1+2\xi_{r-1}+a_r^2<(a_r+1)a_r$, given a growth condition. 

The vector $x$ given by \refeq{x}\ is $\gam$-supported on $(\max(a_r/2a_{r-1},ta_r),2a_r^2a_{r-1})$, so by applying \refeq{20}\ for various  $\lam$ and summing the results, we find that
\begin{equation} \label{22}  |\phi_n(x)|\le 2^{-a_r^2\max(1,t)}\nm x_\gam\le 2^{-a_r^2\max(1,t)}
\cdot ((r-1)(2^{{a}_{r-1}^2}+1))^{L},
\end{equation} where $L=\summ i1{r-1}\lam_i+\mu_i\le (t+1)a_r$,
by \refeq{nmxg}.
 For all $\lam\le{a}_r^2$ it is easy to argue that 
\begin{equation} \label{23}
  |\phi_n(\gam_{\lam a_r}\cdot x)|\le  2^{-a_r^2(\lam+\max(1,t))}\cdot ((r-1)(2^{{a}_{r-1}^2}+1))^{L}.
\end{equation}
Given a growth condition, we
may replace the bounds on the right of \refeq{22}\ and \refeq{23}\ by $2^{-a_r^2\max(1/2,t-1/2)}$ and 
$2^{-a_r^2(\lam+\max(1/2,t-1/2))}$ respectively. Accordingly the vector 
$x'=
x\cdot (\frac r{r+1}2^{{a}_r^2}\gam_{a_r}\up n)^{\lam_r}(r(2^{a_r^2}\gam_{a_r}\up n-1))^{\mu_r}$ from \refeq{x'} 
satisfies
$$|\phi_n(x')|\le 2^{-a_r^2\max(1/2,t-1/2)}\cdot (\frac r{r+1})^{\lam_r}(2r)^{\mu_r}.
$$
Crudely, we may estimate $\mu_r\le \tau\le a_r(t+1)$ (from \refeq{deft} and the
definition of $\tau$), and then we have 
$$2^{-a_r^2\max(1/2,t-1/2)}\cdot (2r)^{\mu_r}\le 2^M,$$ where $$M=(1+\log_2 r)a_r(t+1)-a_r^2\max(1/2,t-1/2).$$ Now a growth condition on the sequence $(a_r)$
will ensure  that $(1+\log_2 r)a_r<a_r^2/12$ for every $r$, and since $(t+1)\le 4\max(1/2,t-1/2)$ for any $t\in\bN_0$, we will then have $$(1+\log_2 r)a_r(t+1)-a_r^2\max(1/2,t-1/2)\le -\frac 23a_r^2\max(1/2,t-1/2).$$  But this equals $$-a_r^2\max(1/3,(2t-1)/3)\le -a_r^2\max(1/3,t/3).$$ So we have 
\begin{equation} \label{24}  |\phi_n(x')|\le 2^{-a_r^2\max(1/3,t/3)}\cdot (\frac r{r+1})^{\lam_r}.
\end{equation}

\blem\l i1eb!
If $w\in\lin\{\gam_i\up n: 0\le i<a_{r+1}\}$,  and $\rho_j\in\bN_0$, for $0\le\rho_j\le 2a_j$ for $j=r+1,...,n$, then writing $y=w\cdot\prodd j{r+1}n(\gam_{a_j}\up n)^{\rho_j}$, we also have 
\begin{equation} \label{21}  \phi_n(y)=
\phi_n(w)\cdot\prodd j{r+1}n 2^{-\rho_ja_j^2} (1-\rho_j/{a}_j)_+,
\end{equation}
where as usual, $t_+$ denotes the maximum $\max(t,0)$.
\elem
\proof
Since $w\in\lin\{\gam_i\up n:0<i<{a}_{r+1}\}$,  we refer to \refeq{phin}\ to find 
$\phi_n(\gam_{i+\summ j{r+1}n\rho_ja_j}\up n)$ for such $i$ (and $\rho_j\le 2{a}_j$), and it is 
$\phi_n(\gam_i\up n)\cdot\prodd j{r+1}n2^{-a_j^2\rho_j}(1-\rho_j/{a}_j)_+$, given the usual growth condition which ensures that, for $\rho_j\le 2a_j$, $i+\summ j{r+1}n\rho_ja_j$ is not in $1+\Lam_n$ unless $i\in\Lam_n$ and all the $\rho_j\le a_j$. Equation \refeq{21} follows.\endproof

\medskip

Equation \refeq{21}\ can also be written as $$\phi_n(w\cdot \prodd  j{r+1}n(2^{a_j^2}\gam_{a_j}\up n)^{\rho_j})=\phi_n(w)\cdot\prodd j{r+1}n (1-\rho_j/{a}_j)_+.$$ So if $\lam_j,\mu_j\le a_j$, 
and we write $$y'=w\cdot \prodd  j{r+1}n(2^{a_j^2}\gam_{a_j}\up n-\gam_0\up n)^{\mu_j}(2^{a_j^2}\gam_{a_j}\up n)^{\lam_j},$$ we will therefore have 
$$
\phi_n(y')=\sum_{{\alp_j=0,\ldots,\mu_j}\atop{\range j{r+1}n}} \phi_n \biggl( w \cdot\prodd j{r+1}n  \biggl( (2^{a_j^2}\gam_{a_j}\up n)^{\lam_j+\alp_j}(-1)^{\mu_j-\alp_j}{{\mu_j}\choose {\alp_j}} 
\biggr) \biggr) $$
\begin{equation} \label{xx1}
  =\sum_{{\alp_j=0,\ldots,\mu_j}\atop{\range j{r+1}n}}\phi_n(w)\cdot\prodd j{r+1}n
(-1)^{\mu_j-\alp_j}{{\mu_j}\choose {\alp_j}} (1-\frac{\alp_j+\lam_j}{a_j})_+.
\end{equation}
We next note that, given a growth condition, the fact that $x$ (as in \refeq{x}) is $\gam$-supported on $[0, 2a_{r}^2a_{r-1})$ tells us that $x'$ (as in \refeq{x'}) is supported on $[0,a_{r+1})$. 
We also recall that $m_0\le 1+\xi_r$, where $m_0$ is  as in \refeq{m0}; so 
$$\mu_j\le m_0\le 1+\xi_r<a_j, \qquad j > r,$$ given a growth condition. 
Let us therefore  apply \refeq{xx1}\ to obtain $\phi_n(\bs^\bi)$ as a multiple of $\phi_n(x')$, using equation \refeq{sixb}:
$$
\phi_n(\bs^\bi)=\phi_n\biggl(x'\cdot\prodd j{r+1}n(j(2^{a_j^2}\gam_{a_j}\up n-\gam_0\up n))^{\mu_j}(\frac j{j+1}2^{a_j^2}\gam_{a_j}\up n)^{\lam_j}\biggr) .
$$ 
But this equals 
$$\sum_{{\alp_j=0,\ldots,\mu_j}\atop{\range j{r+1}n}}\phi_n(x')\cdot\prodd j{r+1}n
(-1)^{\mu_j-\alp_j}{{\mu_j}\choose {\alp_j}} (1-\frac{\alp_j+\lam_j}{a_j})_+,$$ which in turn equals
$$\phi_n(x')\cdot \prodd j{r+1}nj^{\mu_j}(\frac j{j+1})^{\lam_j}\eta_j,\ \textrm{ where }
\eta_j=\summ \alp 0{\mu_j}(-1)^{\mu_j-\alp}{{\mu_j}\choose {\alp}} (1-\frac{\alp+\lam_j}{a_j})_+.$$
 Putting our estimate \refeq{24}\ into this equation, we have 
\begin{equation} \label{etab} 
|\phi_n(\bs^\bi)|\le 2^{-a_r^2\max(1/3,t/3)}\cdot (\frac r{r+1})^{\lam_r}\cdot \prodd j{r+1}nj^{\mu_j}(\frac j{j+1})^{\lam_j}|\eta_j|.
\end{equation}
We can estimate $|\eta_j|$ as follows. If $\lam_j+\mu_j\le a_j$, there is no need to estimate:
 we have 
$\eta_j=\summ \alp 0{\mu_j}(-1)^{\mu_j-\alp}{{\mu_j}\choose {\alp}} (1-\frac{\alp+\lam_j}{a_j})$.
This is $1-\lam_j/a_j$ if $\mu_j=0$, and is $-1/a_j$ if $\mu_j=1$.
It is zero if $\mu_j>1$, by a binomial series argument, or
because second and higher differences of the sequence $(1-(\alp+\lam_j)/a_j)_{\alp=0}^\infty$ are zero. 
If $\lam_j+\mu_j> a_j$, we estimate as follows: we must have $\lam_j\ge a_j-\xi_r$ because 
$\mu_j\le 1+\xi_r$; so $ (1-(\alp+\lam_j)/a_j)_+\le \xi_r/a_j$ for all $\alp\ge 0$; so 
$|\eta_j|\le 2^{\mu_j}\xi_r/a_j\le 2^{1+\xi_r}\xi_r/a_j\le a_j^{-2/3}$, given a growth condition.
We are now in a position to prove:

\blem\l 6.3!
Given growth conditions, one has 
$$\sum_{{\bi\in\cI_1\up n}\atop{m_0(\bi)>1}} |\phi_n(\bs^\bi)|^2\le \prodd j1n\frac {(j+1)^2}{2j+1},
\qquad n\in\bN .$$
\elem

\proof The full sum over $\bi\in\cI_1\up n, m_0(\bi)>1$ is a sum, from $r=1$ to $n$, of contributions
involving $\bi$ with $m_0(\bi)=1+\summ j1r t_ja_j$, and $t_r>0$. For a fixed $r$, we further consider contributions for fixed $\lam_j$ ($\range j1n$), $\mu_j$ ($\range j1r$) and fixed $s=\summ j{r+1}n\mu_j$ (where as usual, $\lam_j=\bi(\frac j{j+1}2^{a_j^2}\gam_{a_j}\up n)$ and $\mu_j=\bi(j(2^{a_j^2\gam_{a_j}\up n-\gam_0\up n}))$. 
So let us write $\cI_1\up {n,r,\lam_1...\lam_n,\mu_1....\mu_r,s}$ for the set of 
$\bi\in\cI_1\up n$ with $m_0(\bi)=1+\summ j1r t_ja_j$, and $t_r>0$, and the given values 
 $\lam_j$ ($\range j1n$), $\mu_j$ ($\range j1r$) and $s=\summ j{r+1}n\mu_j$.

Once these are all fixed, we know the vectors $x$ and $x'$ as in \refeq{x}\ and \refeq{x'}, and also the constants $\tau$, $t$ and $L$ as in \refeq{tau}\ and \refeq{22}. 
For each $j > r$, the values $|\eta_j|$ are determined by $\lam_j$ and $\mu_j$, as described below \refeq{etab}.
Given $\lam_j$, the product $j^{\mu_j}|\eta_j|$ can take the value $1-\lam_j/a_j$ once (when $\mu_j=0$), and the value $j/a_j$ once (when $\mu_j=1$), and values up to $j^{\mu_j}a_j^{-2/3}$ for any $\range {\mu_j}1{1+\xi_r}$. The sum of the squares of all such values is at most $$1+(j/a_j)^2+  (1+\xi_r)j^{2+2\xi_r}a_j^{-4/3}\le 1+a_j^{-1}$$ for all $j>r$, given a growth condition. 
So the sum of the products 
$$\prodd j{r+1}nj^{2 \mu_j}(\frac j{j+1})^{2\lam_j}|\eta_j|^2,$$ for various $\mu_j$ ($j=r+1,...,n$) with $\summ j{r+1}n\mu_j=s$, is at most 
$$\prodd j{r+1}n(\frac j{j+1})^{2\lam_j}(1+a_j^{-1}).$$
 Writing $\cI_1\up\cdot$ for $\cI_1\up {n,r,\lam_1...\lam_n,\mu_1....\mu_r,s}$, we then get 
$$  \sum_{\bi\in\cI_1 \up {\cdot}}|\phi_n(\bs^\bi)|^2
\le 2^{-a_r^2\max(2/3,2t/3)}\cdot (\frac r{r+1})^{2\lam_r}\cdot\prodd j{r+1}n(\frac j{j+1})^{2\lam_j}(1+a_j^{-1}).$$
from \refeq{etab}. 
We can sum this over all  possible $\lam_j$ ($\range jrn$); writing 
$\overline{\cI}_1\up\cdot =\overline{\cI}_1\up{n,r,\lam_1...\lam_{r-1}, \mu_1....\mu_r,s}$ for the union of all sets 
$\cI_1\up {n,r,\lam_1...\lam_n,\mu_1....\mu_r,s}$ as $\lam_j$ varies for $\range j{r}n$, we have 
$$\sum_{\bi\in\overline{\cI}\up {\cdot}}|\phi_n(\bs^\bi)|^2
\le 2^{-a_r^2\max(2/3,2t/3)}\cdot \frac {(r+1)^2}{2r+1}\cdot\prodd j{r+1}n \frac{(j+1)^2}{2j+1}(1+a_j^{-1}).
$$

The number of ways we can choose $\lam_1...\lam_{r-1}, \mu_1....\mu_r,s$ in order to get a nonempty set $\overline{\cI}_1\up{n,r,\lam_1...\lam_{r-1}, \mu_1....\mu_r,s}$ associated with the given value $t$, is less than the number of ways we can pick $2r$ nonnegative integers adding up to an answer $\tau\in[ta_r,(t+1)a_r)$ as in \refeq{tau}. Very crudely, this number is no bigger than $((t+1)a_r)^{2r}$. 
So writing $\overline  {\cI}_1\up{n,r,t}$ for the union of all sets $\overline{\cI}_1\up{n,r,\lam_1...\lam_{r-1}, \mu_1....\mu_r,s}$ such that \refeq{tau}\ holds, we have 
$$
\sum_{\bi\in\overline{\cI}\up {n,r,t}}|\phi_n(\bs^\bi)|^2
\le ((t+1)a_r)^{2a_r}2^{-a_r^2\max(2/3,2t/3)}\cdot \frac {(r+1)^2}{2r+1}\cdot\prodd j{r+1}n \frac{(j+1)^2}{2j+1}(1+a_j^{-1}).
$$
This is dominated by $$2^{-a_r^2\max(1/3,t/3)}\cdot \prodd j{r+1}n \frac{(j+1)^2}{2j+1}(1+a_j^{-1}),$$
given another growth condition. We can of course assume that $\prodd j1\infty (1+a_j^{-1})\le 2$. 
So $$\sum_{\bi\in\overline{\cI}\up {n,r,t}}|\phi_n(\bs^\bi)|^2\le  2^{1-a_r^2\max(1/3,t/3)}\cdot  \prodd j{r+1}n \frac{(j+1)^2}{2j+1}.$$
 Summing over all $t\in\bN_0$ and $r\in[1,n]$ we get 
$$
\sum_{{\bi\in\cI\up n_1}\atop{m_0(\bi)>1}}|\phi_n(\bs^\bi)|^2\le \summ r1n\summ t0\infty   2^{1-a_r^2\max(1/3,t/3)}\cdot   \prodd j{1}n \frac{(j+1)^2}{2j+1}\le \prodd j{1}n \frac{(j+1)^2}{2j+1},
$$ 
given another growth condition. Thus the lemma is proved.
\endproof

\medskip

We can now finish this section by polishing off the case when $\bi\in\cI_1\up n$ but $m_0(\bi)=1$. In that case, since
$\sum_{j\in E(\bi)} \lam_ja_j+\summ j1n\mu_j\le m_0$, we must have $E(\bi)=\emptyset$ and at most one $\mu_j=1$. In fact we must have exactly one $\mu_j=1$, since $0\notin 1+\Lam_n$. So 
$\bs^\bi= \prodd j1n(\frac j{j+1}2^{a_j^2}\gam_{a_j}\up n)^{\lam_j} r(2^{a_r^2}\gam_{1+a_r}\up n-\gam_1\up n)$ for some $r \leq n$, and so, when $\lam_r<{a}_r$, the reader can check that \refeq{phin}\ tells us that 
$$
\phi_n(\bs^\bi)=-\frac r{{a}_r}\prodd j1n (\frac j{j+1})^{\lam_j}\prod_{{j=1,\ldots,n}\atop{j\ne r}}(1-\lam_j/{a}_j).
$$
If $\lam_r={a}_r$, we can safely assume that $\phi_n(\gam_{1+({a}_r+1)a_r+\sum_{j\ne r}\lam_ja_j})=0$, and 
then $\phi_n(\bs^\bi)=0$ in this case. So we have 
$$
\sum_{{\bi\in\cI_1\up n}\atop{m_0(\bi)=1}} |\phi_n(\bs^\bi)|^2\le \summ r1n
\sum_{{\lambda_j=0, ..., {a}_j}\atop{\range j1n}}  (r/{a}_r)^2 \, \prodd j1n (\frac j{j+1})^{\lam_j} ,$$
which is dominated by 
$$\summ r1n (r/{a}_r)^2\cdot\prodd j1n\summ{\lam_j}0\infty (\frac j{j+1})^{2\lam_j}
=
\summ r1n(r/{a}_r)^2\cdot \prodd j1n \frac{(j+1)^2}{2j+1}.$$  
Now 
$\summ r1\infty(r/{a}_r)^2<1$, given a mild growth condition, so $$
\sum_{{\bi\in\cI_1\up n}\atop{m_0(\bi)=1}} |\phi_n(\bs^\bi)|^2\le \prodd j1n \frac{(j+1)^2}{2j+1}, \qquad n\in\bN.$$

Combining this equation with the previous lemma, we have the result:
\begin{theorem} \label{6.4}  Given growth conditions, we have $\sum_{\bi\in\cI_1\up n}|\phi_n(\bs^\bi)|^2<2\cdot \prodd j1n\frac{(j+1)^2}{2j+1}$, for all $n\in\bN$.
\end{theorem}

\section{Bound on $\sum_{\bi\in \cI_2\up n}|\phi_n(\bs^\bi)|^2$} \label{7se}  

To get this bound, we need to have a good estimate of $|\phi_n(\gam_j\up n)|$ in cases when $\gam_j\up n$ is not one of the basis vectors for $H_n$ with respect to which $\phi_n$ is defined directly in \refeq{phin}. That is, we need to know about $\phi_n(\gam_j\up n)$ when $j\ge{a}_{n+1}-{a}_n$.

Let $e_j^*$ ($a_n<j\le a_{n+1}$) denote the linear functional on $H_n$ with $\sprod {e_i}{e_j^*}=\del_{i,j}$.
A general linear functional $\psi=\summ j{1+{a}_n}{{a}_{n+1}}\lam_je_j^*$ 
will have 
\begin{equation} \label{pngk} 
 \psi(\gam_k\up n)=\summ j{1+{a}_n}{{a}_{n+1}}\lam_j2^{-jk}=p(2^{-k}),
 \end{equation}
where $p(t)=\summ j{1+{a}_n}{{a}_{n+1}}\lam_jt^j$ is a polynomial of degree at most ${a}_{n+1}$, with $t^{1+{a}_n}$ a factor 
of $p(t)$. We  will have $p(2^{-k})=\del_{i,k}$  when $0 \leq k < {a}_{n+1}-{a}_n$, if we choose $p=p_{n,i}$, where 
\begin{equation} \label{pnkt}  p_{n,i}(t)=(2^it)^{1+{a}_n}\prod_{{0\le j<{a}_{n+1}-{a}_n}\atop{j\ne i}}\frac
{t-2^{-j}}{2^{-i}-2^{-j}}.
\end{equation}
If we write $\psi_{n,i}$ for the corresponding linear functional, we have 
\begin{equation} \label{psi}
  x=\summ i0{{a}_{n+1}-{a}_n-1}\sprod x{\psi_{n,i}}\gam_i\up n
\end{equation}
for every $x\in H_n$.  This may be seen by checking it on the basis $\seqq i0{{a}_{n+1}-{a}_n-1}{\gam_i\up n}$,
 using \refeq{pngk} and the fact above \refeq{pnkt}. 
Indeed,
\begin{equation} \label{pnik1}  \gam_k\up n=\summ i0{{a}_{n+1}-{a}_n-1}p_{n,i}(2^{-k})\gam_i\up n.
\end{equation}

\blem\l gk!
Given a suitable growth conditions on our underlying sequences, we will have 
$|\phi_n(\gam_l)|\le 2^{-l(1+{a}_n)-{a}_{n+1}^2/3}$ for all $l \in\bN$, 
$l\ge 1+\xi_n$. 
\elem
\bproof It is enough to show this for $l\ge {a}_{n+1}-{a}_n$, since $\phi_n(\gam_l\up n)=0$ for 
$l\in(1+\xi_n,{a}_{n+1}-{a}_n)$ by \refeq{phin}.  Note that 
$\phi_n$ is $\gam$-supported on $[0,1+\xi_n]$. When $k\le 1+\xi_n$, we have 
by \refeq{pnkt}  that $$
p_{n,k}(2^{-l})=
2^{(k-l)(1+{a}_n)}\prod_{{0\le j<{a}_{n+1}-{a}_n}\atop{j\ne k}}\frac{2^{-l}-2^{-j}}{2^{-k}-2^{-j}}.
$$
When $j<k$ the factor $|\frac{2^{-l}-2^{-j}}{2^{-k}-2^{-j}}|$ is in $(1,2]$.  When $j>k$ we have 
$|\frac{2^{-l}-2^{-j}}{2^{-k}-2^{-j}}|\le 2^{k-j+1}$.  Thus 
$$
|p_{n,k}(2^{-l})|\le 2^{(k-l)(1+{a}_n)}\cdot 2^k\cdot \prodd j{k+1}{{a}_{n+1}-{a}_n-1} 2^{k-j+1}
$$
\begin{equation} \label{5-1}
  =  2^{(k-l)(1+{a}_n)}\cdot 2^k\cdot 2^{-\frac 12({a}_{n+1}-{a}_n-k-2)({a}_{n+1}-{a}_n-k-1)}
\le 2^{-1-l(1+{a}_n)-{a}_{n+1}^2/3},
\end{equation}
given a suitable growth condition.
Now the
nonzero coefficients $\phi_n(\gam_k)$ in \refeq{phin} are positive numbers at most 
$2^{-\sum_jt_ja_j^2}$, where $k=1+\sum_j t_ja_j$, $0\le t_j<{a}_j$. These are distinct nonnegative powers of 2, so the sum of all the coefficients is at most 2. So  \refeq{phin}, \refeq{pnik1},
 and \refeq{5-1} give us 
$|\phi_n(\gam_l)|\le 2^{-l(1+{a}_n)-{a}_{n+1}^2/3}$.
\eproof

\begin{theorem} \label{7.2}  Given growth conditions, we have 
$\sum_{\bi\in\cI_2\up n}|\phi_n(\gam_i\up n)|^2\le 2^{-2{a}_{n+1}^2/3}$
for all $n\in\bN$.
\end{theorem}

\begin{proof} If we impose the convolution multiplication on $c_{00}(\bN_0)$, the norm 
$\nm{\summ i0N\bet_ie_i}=\summ i0N2^{-i(1+{a}_n)}|\bet_i|$ is an algebra norm.
One can define an algebra homomorphism $\theta$ from $c_{00}(\bN_0)$ into $H_n$ with 
$\theta(e_i)=\gam_i\up n$, and \reflem{gk}\ can then be rephrased as follows: if $z\in c_{00}$ with 
$z\in\lin\{e_j:j>\xi_n\}$, then $|\phi_n(\theta(z))|\le  2^{-{a}_{n+1}^2/3} \nm z$.

We note that, among the elements of $S_0\up n$,
$ \frac j{j+1}2^{a_j^2}\gam_{a_j}\up n=\theta(u_j)$ with $$\nm{u_j}\le 2^{a_j({a}_j-1-{a}_n)}\le 
2^{-a_j},$$ and $j(2^{a_j^2}\gam_{1+a_j}\up n-\gam_1\up n)=\theta(v_j)$ where 
$$\nm{v_j}\le 2j\cdot 2^{-{a}_n}\le 2^{-{a}_n/2}, \qquad j\le n,$$ given a growth condition. 
If $\bi\in\cI_2\up n$ (so $|\bi|\ge \sqrt{{a}_{n+1}}$),
we  write (as usual) $$\lam_j=\bi(\frac j{j+1}2^{a_j^2}\gam_{a_j}\up n), \; \; \; \; \; \;
\mu_j=\bi (j(2^{a_j^2}\gam_{1+a_j}\up n-\gam_1\up n)).$$  Let $w =\prodd j1n u_j^{\lam_j}v_j^{\mu_j}$.
Then   
$\bs^\bi=\theta(w)$, and 
$\nm w\le 2^{-\summ j1n(\lam_ja_j+\mu_j{a}_n/2)}$.  Also, $$w\in\lin\{e_i:i\ge\sqrt{{a}_{n+1}}\}
\subset \lin\{e_i:i>1+\xi_n\}$$ (given a growth condition).  So 
\reflem{gk}\ applies and tells us (in its ``rephrased'' form) that $$|\phi_n(\bs^\bi)|=|\phi_n(\theta(w))|\le 2^{-{a}_{n+1}^2/3}\nm w
\le 2^{-{a}_{n+1}^2/3-\summ j1n(\lam_ja_j+\mu_ja_n/2)}.$$
So 
$$\sum_{\bi\in\cI_2\up n}|\phi_n(\gam_i\up n)|^2\le 2^{-2{a}_{n+1}^2/3}\cdot
\sum_{{\lam_1,\mu_1,\ldots,\lam_n,\mu_n=0,\ldots,\infty}\atop{\sum_j\lam_j+\mu_j\ge\sqrt{{a}_{n+1}}}} 2^{-2\summ j1n\lam_ja_j+\mu_ja_n/2}.$$ A mild growth condition ensures that the right hand sum is at most $1$ for any $n$, so  we have the required 
result. \end{proof}

 \section{An estimate for $\sum_{\bi\in\cI_3\up n}|\phi_n(\bs^\bi)|^2.$} \label{8se} 

We now turn our attention to the set $\cI_3\up n$, which involves index functions $\bi$ for which 
$\bi(a_k^{-1}e_k)>0$ for some $k\in({a}_n,{a}_{n+1}]$. Since the product of distinct $e_k$ is zero, we get $\bs^\bi=0$ if  $\bi(a_k^{-1}e_k)>0$ for two distinct $k$. If there is one such $k$, and if the index $\bi(a_k^{-1}e_k)=m>0$, we get $$\bs^\bi=a_k^{-m}e_k\cdot \bs^{\bi'}=a_k^{-m}e_k^*(\bs^{\bi'})e_k,$$ where $\bi'$ is an element of $\cI_1\up n\cup\cI_2\up n$. Accordingly we get 
$$\sum_{\bi\in\cI_3\up n}|\phi_n(\bs^\bi)|^2=
\summ k{1+{a}_n}{{a}_{n+1}}\summ m1\infty \sum_{\bi'\in\cI_1\up n\cup\cI_2\up n}
a_k^{-2m}|e_k^*(\bs^{\bi'})\phi_n(e_k)|^2.$$
This is equal to
$$\summ k{1+{a}_n}{{a}_{n+1}}(a_k^2-1)^{-1}\sum_{\bi'\in\cI_1\up n\cup\cI_2\up n}
|e_k^*(\bs^{\bi'})\phi_n(e_k)|^2,$$ which is dominated (since the $e_k^*$ are characters, hence contractive) by $$\summ k{1+{a}_n}{{a}_{n+1}} \, (a_k^2-1)^{-1} \,
\sum_{\bi'\in\cI_1\up n\cup\cI_2\up n} \, \nm{\bs^{\bi'}}_0^2\, |\phi_n(e_k)|^2.$$
The $c_0$ norms of the elements of $S_0\up n$ are listed in the first paragraph of the proof of
Theorem \ref{sac}.
Writing 
$$\eps_j=\nm{\frac j{j+1}2^{a_j^2}\gam_{a_j}\up n}_0\le 2^{-a_j} ,
\; \; \; \; \text{and} \; \; \eps_j'=\nm{j(2^{a_j^2}\gam_{1+a_j}\up n-\gam_1\up n)}_0,$$
we have that $$\eps_j' 
\le j(2^{-a_j} + 2^{-1-a_n}) \leq 2j 2^{-a_j} \le 2^{-{a}_j/2},$$ given a growth condition.
Now $\{\bs^{\bi'}:\bi'\in\cI_1\up n\cup\cI_2\up n\}$ is equal to $$\{\prodd j1n (\frac j{j+1}2^{a_j^2}\gam_{a_j}\up n)^{\lam_j}(j(2^{a_j^2}\gam_{1+a_j}\up n-\gam_1\up n))^{\mu_j}:
\lam_j,\mu_j\in\bN_0,\range j1n\},$$ and so
$$\sum_{\bi'\in\cI_1\up n\cup\cI_2\up n}
\nm{\bs^{\bi'}}_0^2 \le \sum_{\lam_1,\mu_1,\ldots,\lam_n,\mu_n\ge 0}\prodd j1n\eps_j^{2\lam_j}(\eps_j')^{2\mu_j} \le \prodd j1n (1-\eps_j^2)^{-1}(1-{\eps_j'}^2)^{-1}.$$
By the estimates for $\eps_j, \eps_j'$ above, the latter is dominated by
$$\prodd j1n (1-2^{-2{a}_j})^{-1} (1-2^{-{a}_j})^{-1}  <2,$$ given a growth condition.
So, \begin{equation} \label{32}  \sum_{\bi\in\cI_3\up n}|\phi_n(\bs^\bi)|^2\le
\summ k{1+{a}_n}{{a}_{n+1}}2(a_k^2-1)^{-1}|\phi_n(e_k)|^2.
\end{equation}
It is easy to argue  from \refeq{phin} that  
$\sum_j|\phi_n(\gam_j\up n)|\le 2$. Each $e_k=\summ i{1+{a}_n}{{a}_{n+1}}\beta_{k,i}\gam_i\up n$, where $\beta_{k,i}=\sprod{e_k}{\psi_{n,i}}$ as in \refeq{psi}. Now $\sprod{e_k}{\psi_{n,i}}$ is the coefficient of $t^k$ in the polynomial $p_{n,k}(t)$ as in \refeq{pnkt}; a crude estimate is that no coefficient of this polynomial exceeds 
$$2^k(1+{a}_n)\cdot \prod_{{0\le j<{a}_{n+1}-{a}_n}\atop{j\ne k}}\frac 2{|2^{-k}-2^{-j}|}
 \le 2^k(1+{a}_n)\cdot\prodd j0{{a}_{n+1}-{a}_n-1}2^{j+2}\le 2^{k(1+{a}_n)+\frac 12{a}_{n+1}^2} .$$
But this is dominated by $2^{{a}_{n+1}^2},$ given a mild growth condition.
Putting this estimate in \refeq{32}, we have 
$$\sum_{\bi\in\cI_3\up n}|\phi_n(\bs^\bi)|^2\le
\summ k{1+{a}_n}{{a}_{n+1}}2(a_k^2-1)^{-1}2^{{a}_{n+1}^2}.$$
\blem\l 8.1!
Given growth conditions, we have $\sum_{\bi\in\cI_3\up n}|\phi_n(\bs^\bi)|^2\le 1$
for every $n\in\bN$.
\elem 
\proof Given the last inequality, all we need to do is demand the growth conditions ${a}_n\ge n+1$
so that $a_{k}^2>1+2^{2+{a}_{k-1}^2}$ for all $k$. For then, since the $(a_k)$ are strictly increasing, we have $ a_{k}^2>1+2^{r+1+{a}_{k-r}^2}$ for all $0<r<k$.  This implies that 
$$\summ k{1+{a}_n}{{a}_{n+1}}2(a_k^2-1)^{-1}2^{{a}_{n+1}^2}\le
\summ k{2+n}\infty 2(a_k^2-1)^{-1}2^{{a}_{n+1}^2} \le \summ r{1}\infty 2\cdot 2^{-r-1-{a}_{n+1}^2}2^{{a}_{n+1}^2},$$ where 
$r=k-n-1$.  But the last quantity equals  $\summ r{1}\infty 2^{-r}=1$.\endproof

\section{Conclusions}  \label{9se}

\begin{theorem} \label{fin}   
Given growth conditions on the underlying sequence $(a_k)_{k=1}^\infty$, we have $\sum_{\bi\in\cI_0\up n}|\phi_n(\bs^\bi)|^2\le 2+2\cdot\prodd j1n\frac{(j+1)^2}{2j+1}$ for all $n\in\bN$. Lemma \ref{nm22} holds: one has $\nmtn{\gam_0\up n}\le 3\cdot \nmtn{\gam_1\up n}$ for all $n$. The operator norm $\Vert g \Vert^{(n)}_{{\rm op}} \ge \frac 13$. Theorem {\rm \ref{sac}} is true, as is {\rm  \refthm{sec}}, for these choices of the underlying sequence.
\end{theorem}

\proof The first estimate is obtained by summing the estimates given in \refthm{6.1}, \refthm{7.2}\ and 
\reflem{8.1}. Substituting in \refeq{g01}, we have $$(\nmtn{\gam_0\up n})^2\le (\nmtn{\gam_1\up n})^2+2+2\cdot\prodd j1n\frac{(j+1)^2}{2j+1}.$$ Applying the lower estimate \reflem{le}\ for $\nmtn{\gam_1\up n}$, we see that $(\nmtn{\gam_0\up n})^2$ is dominated by $9(\nmtn{\gam_1\up n})^2$, and the second estimate follows. Of course the operator norm $$\opnmn g\ \ge\ \nmtn{\ g\ \gam_0\up n}/\nmtn{\gam_0\up n}=\nmtn{\gam_1\up n}/\nmtn{\gam_0\up n}\ge 1/3.$$ \refthm{sec}\ and \refthm{sac}\ now follow by the argument after \reflem{nm22}. 
\endproof

\medskip

{\em Acknowledgements.} 
  The first author wishes to thank the departments at the Universities of 
Leeds and Lancaster, and in particular the second author and Garth Dales,
for their warm hospitality during a visit in April--May 2013.   
   We also gratefully acknowledge   support from UK research council
grant  EP/K019546/1 for that visit.  
 We thank Garth Dales for a discussion on Banach sequence algebras at that time, and also for providing us with a copy of the related preprint \cite{DU}
in October 2014 after our paper had been accepted
for publication (although he had told us some of the contents of it earlier).   
After one translates our results into their language, as we explained how to do two  paragraphs 
before Theorem \ref{ch}, and compares, it seems that the overlap between our paper and theirs is quite small; for 
example Corollary \ref{finok}
is  also a consequence of  their Proposition 3.1, and we also have examples in Section 1 of weakly compact 
non-Tauberian Banach sequence algebras.   We also thank Joel Feinstein and Garth Dales for pointing out the application of our result mentioned after Theorem \ref{ch} and at the end of Section 6, and for pointing out several typos,
grammatical errors, and small misstatements.


\begin{thebibliography}{99}  

 \bibitem{ABR}  
M. Almus, D. P. Blecher and C. J. Read, {\em Ideals and hereditary subalgebras  in operator algebras},
Studia Math.\  {\bf 212} (2012), 65--93.  


 \bibitem{ABS}  
M. Almus, D. P. Blecher, and S. Sharma,  {\em Ideals and structure of operator algebras,}   J.\
Operator Theory {\bf 67} (2012), 397--436.  


\bibitem{Az}  E. Azoff, and H. Shehada, {\em  Algebras generated by mutually orthogonal idempotent operators,}
  J. Operator Theory {\bf 29} (1993),  249--267.

\bibitem{BHN}  D. P. Blecher, D. M. Hay, and
M. Neal, {\em Hereditary subalgebras of operator algebras,} J.\
Operator Theory {\bf 59} (2008), 333--357.

\bibitem{BLM}  D. P. Blecher
and C.  Le Merdy, {\em Operator algebras and their modules---an
operator space approach,} Oxford Univ.\  Press, Oxford (2004).  



\bibitem{BRI}  D. P. Blecher and C. J. Read, {\em  Operator algebras with contractive approximate identities,}
J. Functional Analysis {\bf 261} (2011), 188--217.  

\bibitem{BRII}  D. P. Blecher and C. J. Read, {\em  Operator algebras with contractive approximate identities, II,} J. Functional Analysis {\bf 261} (2011), 18--217.
 
 \bibitem{BRIII}  D. P. Blecher and C. J. Read, {\em  Operator algebras with contractive approximate identities:
Weak compactness and the spectrum,} J. Functional Analysis
 {\bf  267} (2014), 1837--1850.  
 
 \bibitem{BRS}   D. P. Blecher, Z-J. Ruan and A. M.
Sinclair, {\em A characterization
of operator algebras},  J.\  Functional Analysis {\bf 89} (1990), 288--301.

\bibitem{Dales}  H. G. Dales, {\em Banach algebras and automatic continuity},
London Mathematical Society Monographs.
New Series, 24, Oxford Science Publications.
The Clarendon Press, Oxford University Press,  2000.

\bibitem{DU} H. G. Dales and A. \"Ulger, {\em Approximate identities in Banach function algebras,} Preprint (2014).   

\bibitem{DS} N. Dunford and  J. T. Schwartz, {\em Linear Operators, III: Spectral Operators,}  Wiley Classics Library,
 John Wiley \& Sons, Inc., New York, 1988. 
\bibitem{JF1} J. F. Feinstein,, {\em A note on strong Ditkin algebras,}  Bull. Austral. Math. Soc. 52 (1995),  25--30.

  \bibitem{JF2}  J. F. Feinstein, {\em  Regularity conditions for Banach function algebras,}  In  Function spaces (Edwardsville, IL, 1994), p.\ 117--122,  Lecture Notes in Pure and Appl. Math., 172, Dekker, New York, 1995.

\bibitem{HWW} P. Harmand, D. Werner, and W. Werner,
{\em $M$-ideals in Banach spaces and Banach algebras,}
 Lecture Notes in Math.,  1547, Springer-Verlag, Berlin--New York, 1993.

\bibitem{Hus} T. Husain, {\em Orthogonal Schauder bases,}
  Monographs and Textbooks in Pure and Applied Mathematics, 143, Marcel Dekker, New York, 1991. 


\bibitem{Kap} I. Kaplansky, {\em 
Normed algebras,} 
Duke Math. J. {\bf 16} (1949), 399--418. 

 \bibitem{Kol}  J. J. Koliha, {\em Isolated spectral points,}  Proc. Amer. Math. Soc. {\bf 124} (1996), 3417--3424. 


\bibitem{LN}  K. B. Laursen and M. M.  Neumann, {\em An introduction to local spectral theory,}
 London Mathematical Society Monographs, New Series, 20, The Clarendon Press, Oxford University Press, New York, 2000.



 \bibitem{Pal} T. W. Palmer, {\em Banach algebras and the general
theory of $*$-algebras, Vol.\ I.\ Algebras and Banach algebras,}
Encyclopedia of Math.\ and its Appl., 49, Cambridge University
Press, Cambridge, 1994.  


\bibitem{Read}  C. J. Read,   {\em On the quest for positivity in operator algebras,}
 J. Math. Analysis and Applns.\  {\bf 381} (2011), 202--214.  

\bibitem{Ricker}  W.  Ricker, {\em Operator algebras generated by commuting projections: a vector measure approach,}  Lecture Notes in Mathematics, 1711, Springer-Verlag, Berlin, 1999.
\end{thebibliography}
\end{document}